\IfFileExists{Definitions/mdpi.cls}{%
  \documentclass[mathematics,article,submit,pdftex,moreauthors]{Definitions/mdpi}
}{%
  \documentclass[11pt,a4paper]{article}
\usepackage[T1]{fontenc}
\usepackage[utf8]{inputenc}
\usepackage{lmodern}
\usepackage[margin=25mm]{geometry}
\usepackage{amsmath,amssymb,amsthm,graphicx,xcolor,fancyhdr,float}
\usepackage[hidelinks]{hyperref}
\newtheorem{theorem}{Theorem}[section]
\newtheorem{lemma}[theorem]{Lemma}
\newtheorem{proposition}[theorem]{Proposition}
\newtheorem{corollary}[theorem]{Corollary}
\theoremstyle{remark}
\newtheorem{remark}[theorem]{Remark}
\makeatletter
\newcommand{\firstpage}[1]{\def\@firstpage{#1}}
\newcommand{\pubvolume}[1]{}
\newcommand{\issuenum}[1]{}
\newcommand{\articlenumber}[1]{}
\newcommand{\pubyear}[1]{}
\newcommand{\copyrightyear}[1]{}
\newcommand{\datereceived}[1]{}
\newcommand{\daterevised}[1]{}
\newcommand{\dateaccepted}[1]{}
\newcommand{\datepublished}[1]{}
\newcommand{\hreflink}[1]{}
\newcommand{\Title}[1]{\gdef\paperTitle{#1}}
\newcommand{\TitleCitation}[1]{}
\newcommand{\Author}[1]{\gdef\paperAuthors{#1}}
\newcommand{\AuthorNames}[1]{\hypersetup{pdfauthor={#1}}}
\newcommand{\AuthorCitation}[1]{}
\newcommand{\address}[1]{\gdef\paperAddress{#1}}
\newcommand{\corres}[1]{\gdef\paperCorrespondence{#1}}
\newcommand{\orcidB}{\,\href{https://orcid.org/0000-0001-6992-6977}{\scriptsize[ORCID]}}
\renewcommand{\abstract}[1]{\gdef\paperAbstract{#1}}
\newcommand{\keyword}[1]{\gdef\paperKeywords{#1}}
\AtBeginDocument{%
  \begin{center}
  {\small Manuscript -- 9 October 2026\par}
  \vspace{1em}
  {\LARGE\bfseries\paperTitle\par}
  \vspace{1em}{\large\paperAuthors\par}
  \vspace{0.7em}{\small\paperAddress\par}
  \vspace{0.5em}{\small\paperCorrespondence\par}
  \end{center}
  \vspace{0.5em}\noindent\textbf{Abstract.} \paperAbstract\par
  \vspace{0.5em}\noindent\textbf{Keywords:} \paperKeywords\par\vspace{1em}
}
\makeatother
\pagestyle{fancy}
\fancyhf{}
\fancyhead[L]{\small Tomov and Ganchovski: Safeguarded modAB}
\fancyhead[R]{\small ModAB}
\fancyfoot[C]{\thepage}
\setlength{\headheight}{15pt}
\setlength{\emergencystretch}{2em}
\raggedbottom

}

\usepackage{mathtools}
\usepackage{booktabs}
\usepackage{tabularx}
\usepackage{array}
\usepackage{microtype}
\usepackage{float}
\usepackage{placeins}
\usepackage{xcolor}
\newif\ifshowchanges
\showchangesfalse
\newcommand{\rev}[1]{{\ifshowchanges\color{red}\fi #1}}
\newenvironment{revision}{\begingroup\ifshowchanges\color{red}\fi}{\endgroup}

\providecommand{\institutionalreview}[1]{\par\noindent\textbf{Institutional Review Board Statement:} #1\par}
\providecommand{\informedconsent}[1]{\par\noindent\textbf{Informed Consent Statement:} #1\par}
\providecommand{\dataavailability}[1]{\par\noindent\textbf{Data Availability Statement:} #1\par}
\providecommand{\acknowledgments}[1]{\par\noindent\textbf{Acknowledgments:} #1\par}
\providecommand{\conflictsofinterest}[1]{\par\noindent\textbf{Conflicts of Interest:} #1\par}
\providecommand{\authorcontributions}[1]{\par\noindent\textbf{Author Contributions:} #1\par}
\providecommand{\funding}[1]{\par\noindent\textbf{Funding:} #1\par}

\firstpage{1}
\pubvolume{1}
\issuenum{1}
\articlenumber{0}
\pubyear{2026}
\copyrightyear{2026}
\datereceived{}
\daterevised{}
\dateaccepted{}
\datepublished{}
\hreflink{https://doi.org/}

\Title{Global Convergence and Reliable Stopping in a Safeguarded Modified Anderson--Björck Method}

\TitleCitation{Convergence of a Safeguarded Modified Anderson--Björck Method}

\newcommand{\orcidauthorA}{0000-0003-1902-6473}

\providecommand{\orcidA}{\,\href{https://orcid.org/\orcidauthorA}{\scriptsize[ORCID]}}

\Author{Latchezar Tomov $^{1,*}$\orcidA{} and Nedelcho Ganchovski $^{2}$\orcidB{}}

\AuthorNames{Latchezar Tomov and Nedelcho Ganchovski}

\AuthorCitation{Tomov, L.; Ganchovski, N.}

\address{%
$^{1}$ \quad Department of Computer Science, New Bulgarian University,
21 Montevideo Street, 1618 Sofia, Bulgaria\\
$^{2}$ \quad Proektsoft EOOD, 1164 Sofia, Bulgaria}

\corres{lptomov@nbu.bg}

\abstract{
The modified Anderson–Björck (modAB) method was first proposed in 2023 by Ganchovski and Traykov for solving nonlinear equations in structural mechanics. It was  improved further in 2026 by introducing better switching criteria with account to function symmetry, and bisection fallback based on measuring the bracket reduction rate. Although the algorithm demonstrated superior performance across a great variety of functions, it lacked formal mathematical analysis and proof of convergence. In addition, some discrepancies were noted by L.Tomov in the latest version that need to be addressed: premature application of the Anderson–Björck (AB) correction on the first AB step, caused by retaining the endpoint-update record from the preceding bisection step, and using the corrected endpoint ordinates in symmetry factor calculation instead of the true residuals. The purpose of the present paper is to provide rigorous mathematical proof of global convergence in exact arithmetic, remove the above discrepancies, and provide floating-point safeguards for interpolation, switching and stopping, with explicit treatment of infinite and NaN values, underflow, overflow and round-off errors. These safeguards do not certify the accuracy of the supplied function values. Additional performance improvement was introduced by allowing for early bisection fallback when AB fails to make progress on reducing either the bracket width or the residual, while retaining a finite bound on consecutive exceptions to the prescribed width condition.
}

\keyword{nonlinear equations; root finding; bracketing methods;
Anderson--Björck method; regula falsi; safeguarded interpolation;
adaptive fallback; residual progress; global convergence; floating-point arithmetic; reproducible numerical
experiments}

\begin{document}

\section{Introduction}
\label{sec:introduction}

Finding a zero of a scalar nonlinear function,
\begin{equation}
    f(x)=0,
    \label{eq:root_problem}
\end{equation}
is a fundamental task of numerical analysis.  Bisection is globally
reliable for a continuous function on an interval with endpoint values on the
opposite signs, but it uses only sign information and has linear convergence.
False-position methods exploit endpoint ordinates and can be much faster, but
the classical version may retain one endpoint for many iterations.  The
Illinois, Pegasus, and Anderson--Björck modifications were introduced to
reduce this endpoint-sticking phenomenon
\cite{DowellJarratt1971,DowellJarratt1972,AndersonBjorck1973,King1973}.  Other safeguarded approaches include the
Brent--Dekker family \cite{Brent1971,BusDekker1975}, Ridders' method
\cite{Ridders1979}, and the ITP method \cite{OliveiraTakahashi2021}.

The modified Anderson--Björck method was introduced for nonlinear equations
arising in structural mechanics \cite{GanchovskiTraykov2023}.  A later
implementation combined bisection, an adaptive switching test, corrected
false-position steps, and a fallback controller
\cite{GanchovskiEtAl2026}.  Subsequent code inspection by L.Tomov revealed two issues
that matter both mathematically and computationally.  First, the
Anderson--Björck correction had been applied to variables that were later
treated as if they were the true values of $f$ at the bracket endpoints.
Second, a finite-precision interpolation proposal could participate in a
mixed absolute--relative stopping test before being projected back into the
bracket.  Both issues arise when a quantity used to construct the next
iterate is also used, without the necessary checks, to decide that the
equation has been solved.

The distinction between a function value and an interpolation ordinate is
small in notation but substantial in practice.  Rescaling an ordinate changes the
line from which the next proposal is obtained without changing the function.  Keeping these roles separate lets us first establish that the
interval still contains a zero, and then ask whether its width must tend to zero.
Bisection supplies one part of the answer. A shrinking reference width,
against which each AB phase is tested, supplies the other. A strict width
schedule can, however, interrupt an endpoint that is already approaching a
zero rapidly and slow down the method. We therefore allow a finite number of exceptions when a new
true residual improves on both endpoint residuals by a fixed factor.  The resulting
bound concerns the complete hybrid controller; it is not a new local order
for the classical Anderson--Björck iteration, nor a claim of bisection's
minimax evaluation budget.

\begin{revision}
The finite exception allowance is a significant part of this guarantee.
We exhibit a continuously differentiable function for which the true residual
continues to decrease by more than one half per AB step while the bracket
width tends to a positive limit. This complements the quadratic example which shows that
residual progress explains why an exception can be useful, whereas the
counterexample explains why exceptions must be limited. We also prove
a cubic bound for the switching expression, separate its exact-arithmetic
meaning from floating-point evaluation, and verify the minimal restoration
of the exception cap in a later implementation snapshot.
\end{revision}

\section{The Corrected Hybrid Method}
\label{sec:method}

\subsection{Terminology and the idea of the proof}
\label{subsec:terminology}

We use several standard terms in a very specific sense.

A \emph{bracket} is a closed interval $[a,b]$ for which the endpoint
values have opposite signs:
\[
    f(a)f(b)<0.
\]
For a continuous function, the intermediate value theorem then guarantees at
least one zero between $a$ and $b$.  The word \emph{residual} means simply the
function value $f(x)$ in the real-arithmetic analysis, or $\widehat F(x)$
when a floating-point evaluation is meant.  It is a vertical error in the equation
$f(x)=0$, not automatically a bound on the horizontal error in $x$.

An \emph{ordinate} is the vertical coordinate of a point on a graph.  Thus
$f(a)$ and $f(b)$ are the true endpoint ordinates.  The Anderson--Björck
correction temporarily rescales one ordinate before drawing the next secant.
We call the resulting values \emph{modified ordinates}.  They are interpolation
data, not new values of the function.

An \emph{AB phase} is a consecutive run of Anderson--Björck steps.  A
\emph{fallback} ends such a phase and returns the method to bisection when the
bracket has fallen behind its width schedule and no permitted residual
exception remains. A step is called
\emph{nonexpansive in width} when it cannot make the bracket wider.

The proof answers two questions in order.  First, can the method ever discard
all zeros?  The bracket invariant shows that it cannot.  Second, can the
bracket keep a positive width forever?  Bisection halves it outright, while an
AB phase is compared with a geometrically decreasing reference width.  These
two facts force the width to zero.  The later quantitative argument merely
counts how many iterations may be spent for each guaranteed factor-of-two
reduction.

\subsection{Assumptions and state variables}
\label{subsec:assumptions}

The convergence analysis is carried out in exact real arithmetic under the
following assumptions.

\medskip
\noindent\textbf{Assumption 1.}
The function $f:[a_0,b_0]\rightarrow\mathbb{R}$ is continuous,
$a_0<b_0$, and
\begin{equation}
    f(a_0)f(b_0)<0.
    \label{eq:initial_bracket}
\end{equation}
Unless the method terminates at an exact zero, all endpoint residuals are
nonzero.

\medskip

At global iteration $n$, the method stores a bracket
\[
    I_n=[a_n,b_n],\qquad a_n<b_n,
\]
and the \emph{true endpoint residuals}
\begin{equation}
    \alpha_n=f(a_n),\qquad \beta_n=f(b_n).
    \label{eq:true_residuals}
\end{equation}
During an Anderson--Björck (AB) phase, it also stores
\emph{modified ordinates}
\begin{equation}
    u_n,\qquad v_n.
    \label{eq:working_ordinates}
\end{equation}
They obey
\begin{equation}
    \operatorname{sign}(u_n)=\operatorname{sign}(\alpha_n),
    \qquad
    \operatorname{sign}(v_n)=\operatorname{sign}(\beta_n),
    \qquad
    u_nv_n<0,
    \label{eq:working_signs}
\end{equation}
but their magnitudes may differ from the true residuals.  A new AB phase
starts with
\begin{equation}
    u_n=\alpha_n,\qquad v_n=\beta_n.
    \label{eq:phase_reset}
\end{equation}

The distinction between the actual values of the function in \eqref{eq:true_residuals} and the modified values in \eqref{eq:working_ordinates} is essential. The actual values are used to maintain the interval containing the root, to evaluate the switching criterion, and to assess convergence at termination. The modified values enter only the secant formula for the next approximation.

\subsection{Successive approximations and the enclosing interval}

The method starts with bisection. A bisection step gives the approximation
\begin{equation}
    x_n=\frac{a_n+b_n}{2}.
    \label{eq:bisection_candidate}
\end{equation}
An AB step uses the false-position formula with the modified
ordinates,
\begin{equation}
    x_n
    =
    \frac{a_nv_n-b_nu_n}{v_n-u_n}.
    \label{eq:ab_candidate}
\end{equation}
If $f(x_n)=0$, the method terminates at an exact zero. Otherwise, we
replace by $x_n$ the endpoint whose function value has the same sign as
$f(x_n)$.  If two consecutive AB steps replace the same
endpoint, the modified ordinate at the retained endpoint is multiplied by a
finite positive factor $\theta_n>0$.  The usual Anderson--Björck rule takes
\[
    \theta_n=
    \begin{cases}
    1-f(x_n)/f(e_n),&1-f(x_n)/f(e_n)>0,\\
    1/2,&\text{otherwise},
    \end{cases}
\]
where $e_n$ is the endpoint to be replaced. After two consecutive
replacements on the same side, the ordinate at $e_n$ equals the actual
value of the function there, since that endpoint was introduced at the preceding
step. The implementation may therefore use the stored interpolation
ordinate in this ratio without changing the Anderson--Björck formula.
At the new endpoint, the modified ordinate is set equal to $f(x_n)$.

\subsection{Switching between bisection and Anderson--Björck steps}

When a bisection step is evaluated at the midpoint $m=(a+b)/2$, define the
midpoint value of the endpoint chord,
\begin{equation}
    L(m)=\frac{f(a)+f(b)}{2},
    \label{eq:chord_midpoint}
\end{equation}
and the symmetry factor
\begin{equation}
    r
    =
    1-
    \left|
       \frac{L(m)}{f(b)-f(a)}
    \right|,
    \qquad
    k=r^2.
    \label{eq:symmetry_factor}
\end{equation}
The method may switch from bisection to AB steps when
\begin{equation}
    |f(m)-L(m)|
    <
    k\bigl(|f(m)|+|L(m)|\bigr).
    \label{eq:switch_test}
\end{equation}
The global convergence proof does not depend on this particular switching
criterion. It depends on the condition for returning to bisection, which
we specify below.

\begin{revision}
\begin{lemma}[A cubic bound for the switching expression]
\label{lem:switch_cubic}
Let $f_1$ and $f_2$ be nonzero real endpoint values of opposite signs, and put
\begin{equation}
 D=|f_2-f_1|=|f_1|+|f_2|,\qquad
 y_m=\frac{f_1+f_2}{2},\qquad
 v=\frac{|y_m|}{D},\qquad k=(1-v)^2.
 \label{eq:cubic_definitions}
\end{equation}
For any real midpoint residual $y_3$, let $H=\max(D,|y_3|)$.
Then $0\le v\le1/2$ and, in exact arithmetic,
\begin{equation}
 k|y_m|+k|y_3|\le H(1-v)^2(1+v)\le H.
 \label{eq:cubic_switch_bound}
\end{equation}
\end{lemma}
\begin{proof}
The triangle inequality gives
\[
 |y_m|=\tfrac12|f_1+f_2|
 \le\tfrac12(|f_1|+|f_2|)=D/2.
\]
Since $D>0$, division by $D$ gives $0\le v\le1/2$. Moreover,
\[
 k|y_m|+k|y_3|
 =(1-v)^2(vD+|y_3|)
 \le H(1-v)^2(1+v).
\]
To bound the remaining dimensionless factor, define
\[
 Q(v)=(1-v)^2(1+v)=1-v-v^2+v^3.
\]
Its derivative satisfies
\[
 Q'(v)=3v^2-2v-1=(v-1)(3v+1)<0
 \qquad (0\le v\le1/2).
\]
Thus $Q$ is decreasing, and its endpoint values yield
$3/8=Q(1/2)\le Q(v)\le Q(0)=1$.
Substitution proves \eqref{eq:cubic_switch_bound}.
\end{proof}

The lemma bounds the right-hand side of \eqref{eq:switch_test} in
exact arithmetic. Section~\ref{subsec:switch_rounding} gives the
corresponding rounding-error analysis: first for the distributed
expression when all intermediate results are in the normal range,
and then for a scaled evaluation valid throughout the finite binary64
range. Here $Q(v)$ is dimensionless, whereas $HQ(v)$ has the same
scale as the values of the function.
\end{revision}

Suppose an AB phase starts immediately after a bisection update and the
resulting width is
\begin{equation}
 S=W_c=b-a.
 \label{eq:phase_start_width}
\end{equation}
Within this phase, use local indices: $[a_{t-1},b_{t-1}]$ is the
interval before AB step $t$, $x_t$ is the approximation computed at that
step, and $W_t$ is the interval length after the endpoint replacement.
Thus $t\ge1$ counts completed AB steps within the phase. Choose
$B\ge1/2$, an integer $M\ge0$, and a residual factor $0<\eta<1$.
Initialize a counter $c=0$ and a width threshold $BS$. At step $t$ the
threshold used in the comparison is
\begin{equation}
 T_t=BS\,2^{-(t-1)}=\frac{2BS}{2^t}.
 \label{eq:adaptive_threshold}
\end{equation}
Before replacing an endpoint, denote the smaller of the two endpoint
residual magnitudes by
\begin{equation}
 R_t^-=\min\{|f(a_{t-1})|,|f(b_{t-1})|\}.
 \label{eq:preupdate_residual}
\end{equation}
This minimum is taken over the actual values of the function. After replacing
the endpoint, let $W_t$ be the length of the resulting interval and apply
the following conditions in order:
\begin{enumerate}
 \item If $W_t\le T_t$, continue with AB steps and set $c=0$.
 \item Else, if
 \begin{equation}
 |f(x_t)|<\eta R_t^-\quad\hbox{and}\quad c<M,
 \label{eq:residual_exception}
 \end{equation}
 continue with AB steps and increase $c$ by one.
 \item Else end the AB phase and take a bisection step next. Restore
 the modified ordinates to the actual endpoint values and discard the
 record of the endpoint replaced in the preceding AB step.
\end{enumerate}
Whenever the phase continues, halve the threshold for the next step.
Allowing an exception on the basis of residual decrease does not restart
this sequence of thresholds. Thus the counter records consecutive
exceptions to the width condition; it returns to zero when that condition
is satisfied or a new AB phase begins.

In the numerical experiments we use
\begin{equation}
 B=2,\qquad M=3,\qquad\eta=\tfrac12.
 \label{eq:adaptive_parameters}
\end{equation}
The first two thresholds, $2S$ and $S$, are always satisfied because the
interval lengths are nonincreasing. The first possible violation of the
width condition occurs after the third AB step, when the threshold is
$S/2$. At most three consecutive violations are permitted, and each
requires \eqref{eq:residual_exception}. If steps 3, 4 and 5 all require residual exceptions, then the allowance
is exhausted before step 6. In that case, the sixth AB step must find an
exact zero or satisfy the width condition; else the method returns
to bisection. If an intervening step satisfies the width condition, the
counter returns to zero and later exceptions remain possible. Thus six
is not an upper bound on the number of steps in an AB phase.

\begin{revision}
The implementation parameter \texttt{MaxResidualSteps} represents the
bound $M$ on consecutive exceptions, rather than a bound on the number of
steps in an AB phase. A phase can continue for longer if the width
condition is subsequently satisfied. The separate parameter
\texttt{MaxIterations} bounds the total number of iterations and reaching it
may produce \texttt{MaxIterationsExceeded}. Such termination neither
forces a bisection step nor implies a bound on the length of the interval.
Proposition~\ref{prop:uncapped_counterexample} shows why a bound on the
total number of iterations cannot replace the bound on consecutive
exceptions.
\end{revision}

For the convergence estimates, define
\begin{equation}
 C=2^{M+1}B.
 \label{eq:effective_C}
\end{equation}
Lemma~\ref{lem:adaptive_envelope} proves that an active phase satisfies
$W_t\le CS/2^t$, even when individual width conditions are violated. Thus
\eqref{eq:adaptive_parameters} gives
\begin{equation}
 C=32=2^5.
 \label{eq:C32}
\end{equation}
The same upper bound is obtained for the fixed-delay version by setting
$B=16$ and $M=0$. In that case the first five width conditions hold
automatically, and the first comparison with $S/2$ occurs after the sixth
AB step. The earlier convention of multiplying the width before the
switching bisection by $16$ gives the same bound in exact arithmetic,
since that bisection halves the interval length. This equivalence in
exact arithmetic, however, does not imply bitwise agreement between floating-point
implementations.

\section{Preservation of the Enclosing Interval}
\label{sec:invariants}

\begin{lemma}[Interiority of the AB approximation]
\label{lem:interiority}
Let $a<b$ and $uv<0$.  Then
\[
    x=\frac{av-bu}{v-u}
\]
belongs to $(a,b)$.
\end{lemma}

\begin{proof}
Assume without loss of generality that $u<0<v$.  Then
\[
    x
    =
    \frac{av+b|u|}{v+|u|}
    =
    \lambda a+(1-\lambda)b,
    \qquad
    \lambda=\frac{v}{v+|u|}\in(0,1).
\]
Thus $x$ is a strict convex combination of $a$ and $b$ and belongs to
$(a,b)$. The case $v<0<u$ follows in the same way.
\end{proof}

\begin{lemma}[Preservation of the signs of the modified ordinates]
\label{lem:working_signs}
Suppose \eqref{eq:working_signs} holds.  If the retained modified ordinate is
multiplied by a finite factor $\theta>0$ and the newly inserted endpoint
receives the actual function value at the new point, then
\eqref{eq:working_signs} holds at the next nonterminating AB step.
\end{lemma}

\begin{proof}
Multiplication by a positive factor preserves the sign of the retained
ordinate. At the new endpoint, the modified ordinate equals the actual value of the function. Since the new interval has endpoint values of opposite
signs, the modified ordinates also have opposite signs.
\end{proof}

\begin{proposition}[Nested intervals containing a zero]
\label{prop:bracket}
At every nonterminating iteration,
\begin{equation}
    f(a_n)f(b_n)<0,
    \qquad
    I_{n+1}\subseteq I_n.
    \label{eq:nested_brackets}
\end{equation}
Consequently, every $I_n$ contains at least one zero of $f$, and the widths
\begin{equation}
    W_n=b_n-a_n
    \label{eq:width}
\end{equation}
are nonincreasing.  A bisection step gives
\begin{equation}
    W_{n+1}=\frac{1}{2}W_n.
    \label{eq:bisection_halving}
\end{equation}
\end{proposition}

\begin{proof}
The sign condition holds initially by Assumption~1. The midpoint belongs
to the current interval, as does the AB approximation by
Lemma~\ref{lem:interiority}. Unless this new point is an exact zero,
precisely one of the two resulting subintervals has endpoint values of
opposite signs. The method selects this subinterval, which contains a
zero by continuity and the intermediate value theorem. This proves the
nesting property and the monotonicity of the interval lengths. A
bisection step gives a subinterval of half the preceding length.
\end{proof}

\subsection{Residual decrease during an exceptional AB step}

\begin{lemma}[Decrease of the smaller endpoint residual]
\label{lem:residual_progress}
Let $R^-=\min(|f(a)|,|f(b)|)>0$ before a nonterminating step within
the interval, and let
$R^+$ be the corresponding minimum after the sign-preserving update.
If the new residual satisfies $|f(x)|<\eta R^-$ with $0<\eta<1$, then
\begin{equation}
 R^+=|f(x)|<\eta R^-.
 \label{eq:residual_contraction}
\end{equation}
For $k$ consecutive such steps, the final minimum is less than
$\eta^k$ times its initial value.
\end{lemma}
\begin{proof}
The retained endpoint has residual magnitude at least $R^-$. The newly
inserted endpoint has residual magnitude below $\eta R^-<R^-$. It therefore
becomes the endpoint with the smaller residual magnitude, which proves
\eqref{eq:residual_contraction}. Repeated substitution proves the last claim.
\end{proof}

The minimum must be taken before replacing the endpoint. After the
replacement it would be no greater than $|f(x)|$, making the strict
contraction inequality impossible. In exact arithmetic, replacing $f$ by
$cf$, $c\ne0$, leaves the condition unchanged, since both sides are
multiplied by $|c|$. Residual decrease alone, however, does not give a
bound on the error in the root approximation. Near a zero of multiplicity
$m$, the model $f(x)\simeq c(x-\alpha)^m$ shows that reducing the residual
by a factor $\eta$ reduces the distance to the zero by approximately
$\eta^{1/m}$. The bound on the interval length and the finite number of
permitted exceptions are therefore still needed for the geometric
convergence estimate. In general, neither AB steps nor bisection steps
make the smaller endpoint residual decrease monotonically; the lemma
applies only when its contraction hypothesis is satisfied.

\subsection{One-sided convergence without contraction of the interval}

\begin{proposition}[A one-sided quadratic AB phase]
\label{prop:one_sided}
For the continuously differentiable function
\begin{equation}
 f(x)=\begin{cases}x,&x\le0,\\x+x^2,&x>0,\end{cases}
 \label{eq:one_sided_function}
\end{equation}
an AB phase starting from $[-1,1]$ with the actual endpoint values as its
initial ordinates has a fixed
right endpoint. If interpolation is continued without a width fallback,
the left-endpoint error converges quadratically to zero, whereas the bracket
width tends to one.
\end{proposition}
\begin{proof}
Write the left endpoint as $-e$ and the right modified ordinate as $v>1$.
The actual left ordinate is $-e$. The next secant approximation is $-e'$, where
\[
 e'=\frac{e(v-1)}{v+e}>0.
\]
Starting with $e=1,v=2$, the first two errors are $1/3$ and $1/7$.
\rev{The second consecutive update of the left endpoint activates the AB correction.} Whenever it is applied,
\[
 v'=v\left(1-\frac{e'}e\right)
    =\frac{v(1+e)}{v+e}=1+e'.
\]
Thus the next trial point remains to the left and, from the error $1/7$ onward,
\begin{equation}
 e_{k+1}=\frac{e_k^2}{1+2e_k}.
 \label{eq:one_sided_quadratic}
\end{equation}
This positive sequence decreases to zero and
$e_{k+1}/e_k^2\to1$. Its first subsequent values are $1/63$, $1/4095$ and
$1/16777215$. The right endpoint stays at one, so $W_k=1+e_k\to1$.
\end{proof}

This phase also occurs within the hybrid method. Starting from $[-3,1]$,
the first midpoint is $-1$, and the switching inequality is
$0.5<0.81(1+0.5)$. The method therefore enters the AB phase just as described.
Here the left endpoint converges quadratically even though the interval
length tends to a positive limit. This is a property of this particular
phase only.
With only finitely many consecutive exceptions permitted, the method
eventually returns to bisection unless it has already satisfied a stopping
condition.

\begin{revision}
\subsection{Necessity of a bound on consecutive residual exceptions}
\label{subsec:cap_counterexample}

In the preceding example, continued interpolation rapidly improves one
endpoint approximation. The next example shows why residual decrease
cannot be the only condition for continuing an AB phase: even a uniform
reduction by more than one half at every subsequent step need not force
the interval length to zero.

\begin{proposition}[Residual contraction without width convergence]
\label{prop:uncapped_counterexample}
Consider the usual AB correction, the switching test
\eqref{eq:switch_test}, and the rule with $B=2$, $\eta=1/2$, but with
no bound on the number of residual exceptions. For
\begin{equation}
 f(x)=\begin{cases}-x^2,&x<0,\\7x^2,&x\ge0,\end{cases}
 \qquad [a_0,b_0]=[-3,1],
 \label{eq:cap_counterexample}
\end{equation}
the exact-arithmetic process enters an AB phase after its first bisection
and never leaves it. The left endpoints converge to the unique
zero, while the right endpoint remains $1$ and the width tends to $1$.
From the third AB step onward, the ratio of the new residual magnitude to
the smaller preceding endpoint residual magnitude is at most $0.49<1/2$.
The function belongs to $C^1(\mathbb R)$.
\end{proposition}

\begin{proof}
The only zero is $0$. On the two open half-lines the derivatives are $-2x$
and $14x$, respectively. At zero,
$f'(0)=\lim_{h\to0}f(h)/h=0$, and both one-sided derivatives tend to zero.
This proves the stated regularity.

Initially $f(-3)=-9$ and $f(1)=7$. The first midpoint is $-1$, with
$f(-1)=-1$, and the chord midpoint value is also $(-9+7)/2=-1$.
Consequently, the left-hand side of \eqref{eq:switch_test} is zero and
its right-hand side is positive. The method enters AB on $[-1,1]$.

Let $j$ count AB steps after this entry. Write the left endpoint as $-e_j$
and the retained right modified ordinate as $v_j$, with $e_0=1$ and $v_0=7$.
The actual function value at the left endpoint is $-e_j^2$. If the next trial point remains to the
left of the zero, the secant formula gives
\begin{equation}
 e_j=\frac{e_{j-1}v_{j-1}-e_{j-1}^2}
              {v_{j-1}+e_{j-1}^2}.
 \label{eq:cap_secant_error}
\end{equation}
There is no retained-ordinate rescaling on the first AB step. Hence
\[
 e_1=\frac{7-1}{7+1}=\frac34,\qquad v_1=7,
 \qquad e_2=\frac{7e_1-e_1^2}{7+e_1^2}=\frac{75}{121}.
\]
Both approximations lie strictly to the left of zero and are closer to
it than the preceding left endpoint.

For two successive updates of that endpoint, the AB multiplier is
$1-e_j^2/e_{j-1}^2$, which is positive whenever $0<e_j<e_{j-1}$.
Rearranging \eqref{eq:cap_secant_error} gives
\[
 (e_{j-1}-e_j)v_{j-1}=e_{j-1}^2(1+e_j).
\]
Thus, for $j\ge2$ while these inequalities hold,
\begin{align}
 v_j
 &=v_{j-1}\left(1-\frac{e_j^2}{e_{j-1}^2}\right)\notag\\
 &=\frac{e_{j-1}^2(1+e_j)}{e_{j-1}-e_j}
   \frac{(e_{j-1}-e_j)(e_{j-1}+e_j)}{e_{j-1}^2}\notag\\
 &=(1+e_j)(e_{j-1}+e_j).
 \label{eq:cap_working_ordinate}
\end{align}
Applying this identity at $j=2$ and then taking the next secant yields
\[
 e_3=\frac{146325}{384296},\qquad
 \rho_3:=\frac{e_3}{e_2}=\frac{1951}{3176}.
\]
In particular, $0<e_3<2/5$ and $1/2<\rho_3<7/10$.

To prove that all subsequent approximations have the same properties, put $\rho_j=e_j/e_{j-1}$.
Using \eqref{eq:cap_working_ordinate} in the next secant, we obtain
\begin{align*}
 \frac{e_{j+1}}{e_j}
 &=\frac{v_j-e_j}{v_j+e_j^2}\\
 &=\frac{(1+e_j)e_{j-1}+e_j^2}
         {(1+e_j)e_{j-1}+e_j+2e_j^2}\\
 &=\frac{1+e_j(1+\rho_j)}
         {1+\rho_j+e_j(1+2\rho_j)}.
\end{align*}
The last equality divides numerator and denominator by $e_{j-1}>0$.
Define
\[
 \Phi(e,\rho)=\frac{1+e(1+\rho)}{1+\rho+e(1+2\rho)}.
\]
For $e\ge0$ and $\rho>0$ its denominator is positive, and direct
differentiation gives
\[
 \frac{\partial\Phi}{\partial\rho}
 =-\frac{(1+e)^2}{[1+\rho+e(1+2\rho)]^2}<0,\qquad
 \frac{\partial\Phi}{\partial e}
 =\frac{\rho^2}{[1+\rho+e(1+2\rho)]^2}>0.
\]
On $0\le e\le2/5$, $1/2\le\rho\le7/10$, monotonicity therefore gives
\begin{equation}
 \frac{10}{17}=\Phi(0,7/10)
 \le\Phi(e,\rho)\le\Phi(2/5,1/2)=\frac{16}{23}<\frac7{10}.
 \label{eq:cap_invariant_region}
\end{equation}
Both bounds belong to $[1/2,7/10]$. Since the required inequalities hold
for $j=3$, induction gives
\[
 0<e_{j+1}\le\tfrac7{10}e_j,\qquad
 \rho_{j+1}\in[1/2,7/10].
\]
These inequalities simultaneously prove that the next trial point remains
strictly left of zero, that the AB multiplier stays positive, and that
the hypotheses are preserved. Thus $e_j\to0$ and the right endpoint
is never replaced.

The actual value of the function at the right endpoint is $7$ and the left residual magnitude is
$e_j^2\le1$. Before AB step $j$ the minimum residual magnitude is therefore
$e_{j-1}^2$. For every $j\ge3$,
\begin{equation}
 \frac{|f(-e_j)|}{\min\{|f(-e_{j-1})|,|f(1)|\}}
 =\rho_j^2\le(7/10)^2=0.49<\tfrac12.
 \label{eq:cap_residual_ratio}
\end{equation}

At the beginning of the phase the width is $S=2$, so the threshold at AB step $j$ is
$T_j=4/2^{j-1}$. The first two interval lengths satisfy the respective bounds
$4$ and $2$. For $j\ge3$, however, $T_j\le1$ while
$W_j^{\rm AB}=1+e_j>1$. The width condition is therefore violated at every subsequent step,
whereas \eqref{eq:cap_residual_ratio} permits a residual exception at each
of these steps. Without a bound on their number, the AB phase continues
indefinitely and the interval length tends to $1$.
\end{proof}

The same example shows that the earlier global estimate cannot be
retained after removing the bound on consecutive exceptions. If $N$ counts all bracket updates,
including the initial bisection, then $W_0=4$ and
\begin{equation}
 W_{14}=1+e_{13}>1
       =W_0\,2^{-\lfloor14/7\rfloor}.
 \label{eq:cap_global_violation}
\end{equation}
The approximations themselves still converge to the zero. What fails is
convergence of the whole bracket and, consequently, width-based termination
at a tolerance below the limiting width. For the mixed tolerance used here,
any fixed $\varepsilon_{\rm abs}<1$ and finite $\varepsilon_{\rm rel}$
leave the width test unsatisfied for all sufficiently late steps, since the
candidate tends to zero. In particular, the tolerances of the numerical
experiment in Section~\ref{subsec:native_cap} cannot terminate this
exact-arithmetic process by width. It never encounters an exact zero.
\end{revision}

\subsection{A geometric bound on the interval length during an AB phase}

\begin{lemma}[Geometric upper bound for an AB phase]
\label{lem:adaptive_envelope}
Let a phase start with width $S$, use \eqref{eq:adaptive_threshold}, and permit
at most $M$ consecutive width-test exceptions while active. If it is
still active after $t$ steps, then
\begin{equation}
 W_t\le\frac{CS}{2^t},\qquad C=2^{M+1}B.
 \label{eq:fallback_threshold}
\end{equation}
\end{lemma}
\begin{proof}
If the width condition is satisfied at step $t$, then $W_t\le T_t$,
which is stronger than \eqref{eq:fallback_threshold}. Otherwise, let
$j=t-c$ be the last index at which it was satisfied. Since the phase
continues, $c\le M$. If the condition has not yet been satisfied, take
$j=0$ and set $T_0=2BS$; then $W_0=S\le T_0$ because $B\ge1/2$.
The bracket widths are nonincreasing, so in either case
\[
 W_t\le W_j\le T_j=2^cT_t
       \le 2^M\frac{2BS}{2^t}=\frac{CS}{2^t}.
\]
Whenever the width condition is satisfied, the counter returns to zero
and the estimate $W_j\le T_j$ holds at that index. The same argument
therefore applies after any number of such resets.
\end{proof}

The proof depends on the bound on consecutive exceptions, but not on
the value of $\eta$. The residual condition determines when an exception
is permitted and the bound on their number gives the estimate for the
length of the interval. If an AB step leads to a return to bisection,
\eqref{eq:fallback_threshold} need not hold after that last AB step.
Both this step and the following bisection are included explicitly in
the iteration bound derived below.

\begin{revision}
The result holds for every finite integer $M\ge0$, not only for $M=3$.
For fixed $B$, increasing the number of permitted consecutive exceptions
by one doubles the constant $C=2^{M+1}B$. With $B=2$, the iteration bound
below uses $q=M+4$. In particular, $M=3$ gives $q=7$; every other finite
choice also gives an upper bound on the  length of the interval.
Proposition~\ref{prop:uncapped_counterexample} shows that this estimate
cannot be retained if arbitrarily many consecutive exceptions are
permitted. It does not single out three as the necessary or optimal
finite bound.
\end{revision}

\section{Global Convergence}
\label{sec:convergence}

The preceding results show that every interval contains a zero. It
remains to prove that the interval lengths tend to zero. We first prove
convergence and then estimate the number of iterations needed for a
given reduction in length. For the latter estimate, we must also account
for a final AB phase that has not yet ended.

Throughout this section, $W_n$ denotes the interval length after $n$
completed endpoint replacements, with $W_0=b_0-a_0$. We study the
iteration without a bound on its length and without a positive stopping
tolerance. The process may still terminate at an exact zero. In the
implementation, termination at an exact zero or upon satisfying the
prescribed width tolerance already gives the required result.

\subsection{Convergence of the enclosing intervals}

\begin{theorem}[Global convergence]
\label{thm:global_simple}
Under Assumption~1, with the actual function values and modified ordinates
kept separate and the adaptive rule of Section~\ref{sec:method}, the
exact-arithmetic iteration, continued without an iteration limit, either
terminates at an exact zero or satisfies
\[
    W_n\longrightarrow0.
\]
In the latter case, the nested brackets shrink to a single point $x^\ast$, and that point satisfies
$f(x^\ast)=0$.
\end{theorem}

\begin{proof}
We distinguish two cases.

\emph{Case 1: infinitely many bisection steps.}
Each bisection halves the current interval length, and no intervening AB
step increases it. After $j$ bisections the length is therefore at most
$W_0/2^j$, which tends to zero.

\emph{Case 2: finitely many bisection steps.}
After the last bisection, all subsequent steps belong to a single AB
phase. Indeed, ending that phase would require another bisection, contrary
to the choice of the last one.

Let $W_c$ be the width at the start of that phase and let $W(t)$ be the width
after $t$ AB steps.  Because the phase is still active, Lemma~\ref{lem:adaptive_envelope} gives
\[
    W(t)\leq \frac{C W_c}{2^t}.
\]
The right-hand side tends to zero, and hence so does $W(t)$.

Thus $W_n\to0$ in either case.  By
Proposition~\ref{prop:bracket}, the closed intervals $I_n$ are nonempty,
nested, and contain a zero.  A standard nested-interval argument now shows that
their intersection consists of one point:
\[
    \bigcap_{n=0}^{\infty} I_n=\{x^\ast\}.
\]
To verify that this point is a zero, choose any zero $z_n\in I_n$.  Both
$z_n$ and $x^\ast$ lie in the same interval, so
\[
    |z_n-x^\ast|\leq W_n\longrightarrow0.
\]
Thus $z_n\to x^\ast$.  Since $f(z_n)=0$ and $f$ is continuous,
$f(x^\ast)=0$.
\end{proof}

\begin{remark}[Scope of the convergence argument]
\label{rem:proof_independence}
The argument does not assume differentiability, a unique zero, a particular
local convergence order, or the specific switching test
\eqref{eq:switch_test}.  It uses only intervals whose endpoint values have opposite signs,
interpolation points within these intervals, nonincreasing interval
lengths, the geometric upper bound during each continuing AB phase, and
the bisection step that follows the end of a phase.
\end{remark}

\subsection{An upper bound on the number of iterations}

To obtain a quantitative convergence estimate, define
\begin{equation}
    d=\lceil\log_2 C\rceil,
    \qquad
    q=d+2.
    \label{eq:dq}
\end{equation}
The integer $d$ is chosen so that $C\leq 2^d$. The additional two
iterations in $q$ account for the last AB step of a phase and the
bisection step that follows it.

We say that a sequence of iterations gives $h$ \emph{halvings} if its
final interval length, relative to the initial length $S$, is at most
\[
    \frac{S}{2^h}.
\]
This terminology describes the reduction in interval length; it does not
require the sequence to contain $h$ actual bisection steps.

A \emph{completed phase block} consists of an AB phase with $t\geq1$ AB
steps, followed by the first bisection after fallback.  Its total length is
$\ell=t+1$.

\begin{lemma}[Guaranteed shrinkage of a completed AB phase]
\label{lem:completed_phase}
Let a completed phase block start with width $S$.  At its end the width is at
most $S/2^h$, where
\begin{equation}
    h=
    \begin{cases}
       1, & t\leq d+1,\\
       t-d, & t\geq d+2.
    \end{cases}
    \label{eq:phase_credit}
\end{equation}
In both cases,
\begin{equation}
    \ell\leq qh.
    \label{eq:length_credit}
\end{equation}
\end{lemma}

\begin{proof}
We consider the two ranges of $t$ separately.

If $t\leq d+1$, the AB steps do not increase the interval length, and the
following bisection halves it. Thus $h=1$, and
\[
    \ell=t+1\leq d+2=q=qh.
\]

Now suppose that $t\geq d+2$.  The phase was still active after step $t-1$.
Lemma~\ref{lem:adaptive_envelope} therefore gives
\[
    W(t-1)
    \leq
    \frac{CS}{2^{t-1}}
    \leq
    \frac{S}{2^{t-d-1}},
\]
where the last inequality uses $C\leq2^d$.  The $t$th AB step cannot increase
the width, and the following bisection divides it by two.  Hence
\[
    W_{\mathrm{end}}
    \leq
    \frac{S}{2^{t-d}}.
\]
The block therefore guarantees $h=t-d$ halvings.  Since $h\geq2$ and
$t=d+h$,
\[
    \ell=t+1=d+h+1\leq(d+2)h=qh.
\]
\end{proof}

Table~\ref{tab:phase_accounting} summarizes the two estimates. For a short
AB phase, the following bisection alone gives the required reduction.
For a long phase, the geometric bound already implies a stronger
reduction before that bisection.

\begin{table}[H]
\caption{Shrinkage guaranteed by a completed phase block.}
\label{tab:phase_accounting}
\centering
\begin{tabular}{@{}lll@{}}
\toprule
AB steps $t$ & Guaranteed halvings $h$ & Iteration cost \\ \midrule
$t\leq d+1$ & $1$ & $t+1\leq d+2=q$\\
$t\geq d+2$ & $t-d$ & $t+1\leq q(t-d)$\\
\bottomrule
\end{tabular}
\end{table}

The first $n$ iterations may end within an AB phase or just after its
last AB step, before the following bisection. We therefore also need an
estimate for this possibly incomplete final part.

\begin{lemma}[An incomplete final AB phase]
\label{lem:tail}
Suppose the first $n$ iterations end either within an AB phase or
immediately after its last AB step, before the following bisection. Let $S$ be the interval length at the
start of this final AB phase. If the unfinished part has length
$\ell_{\mathrm{tail}}$, then
there is an integer $h_{\mathrm{tail}}\geq0$ such that its final width is at
most $S/2^{h_{\mathrm{tail}}}$ and
\begin{equation}
    \ell_{\mathrm{tail}}
    \leq
    qh_{\mathrm{tail}}+(q-1).
    \label{eq:tail_bound}
\end{equation}
\end{lemma}

\begin{proof}
Let $t$ be the number of AB steps in the unfinished part.

If the phase is still active and $t\leq d$, take
$h_{\mathrm{tail}}=0$.  Then
$t\leq d=q-2<q-1$.  If the phase is still active and $t\geq d+1$, Lemma~\ref{lem:adaptive_envelope} gives
\[
    W(t)\leq\frac{CS}{2^t}\leq\frac{S}{2^{t-d}}.
\]
Thus we may take $h_{\mathrm{tail}}=t-d$, and
$t=d+h_{\mathrm{tail}}\leq qh_{\mathrm{tail}}$.

It remains to consider a final AB step that has just triggered fallback.  If
$t\leq d+1$, again take $h_{\mathrm{tail}}=0$; then
$t\leq d+1=q-1$.  If $t\geq d+2$, the phase was active after step $t-1$, so
\[
    W(t)\leq W(t-1)
    \leq \frac{CS}{2^{t-1}}
    \leq \frac{S}{2^{t-d-1}}.
\]
Now take $h_{\mathrm{tail}}=t-d-1$.  Since
$t=h_{\mathrm{tail}}+d+1$,
\[
    t\leq qh_{\mathrm{tail}}+(q-1).
\]
These cases cover every possible unfinished final part.
\end{proof}

\begin{theorem}[A global bound on the interval length]
\label{thm:quantitative}
Under the assumptions of Theorem~\ref{thm:global_simple}, after any $n$
completed bracket updates, unless an exact zero was found earlier,
\begin{equation}
    \boxed{
    W_n
    \leq
    W_0\,2^{-\lfloor n/q\rfloor},
    \qquad
    q=\lceil\log_2 C\rceil+2.
    }
    \label{eq:global_width_bound}
\end{equation}
\end{theorem}

\begin{proof}
Partition the first $n$ iterations into:

\begin{itemize}
    \item bisection steps not already included in a completed phase block;
    \item completed phase blocks from Lemma~\ref{lem:completed_phase};
    \item at most one unfinished final part from Lemma~\ref{lem:tail}.
\end{itemize}

Let $H$ be the sum of the numbers of halvings given by these parts.
A single bisection has length $1$ and gives one halving. Each completed
phase block of length $\ell$ satisfies $\ell\leq qh$, while the
incomplete final part satisfies the bound in Lemma~\ref{lem:tail}.
Adding these inequalities gives
\begin{equation}
    n\leq qH+(q-1).
    \label{eq:amortized_sum}
\end{equation}

Because $H$ is an integer, this implies
\[
    H\geq\left\lfloor\frac{n}{q}\right\rfloor.
\]
Indeed, if $H\leq\lfloor n/q\rfloor-1$, then
\[
    n
    \leq
    q\left(\left\lfloor\frac{n}{q}\right\rfloor-1\right)+(q-1)
    <
    q\left\lfloor\frac{n}{q}\right\rfloor
    \leq n,
\]
which is impossible.  The accumulated halvings give
$W_n\leq W_0/2^H$, and the stated bound follows.
\end{proof}

\begin{corollary}[The adaptive method with $B=2$ and $M=3$]
\label{cor:C32}
For the adaptive method in exact arithmetic with $B=2$, $M=3$ and $\eta=1/2$,
\[
    C=32,\qquad d=5,\qquad q=7,
\]
and therefore
\begin{equation}
    \boxed{
    W_n\leq W_0\,2^{-\lfloor n/7\rfloor}.
    }
    \label{eq:C32_bound}
\end{equation}
\end{corollary}

Figure~\ref{fig:threshold} shows the common geometric upper bound used
in the proof, which is also the fixed-delay threshold. In the fixed-delay version, the first
five width conditions hold automatically. In the adaptive version, only
the first two hold automatically; up to three consecutive exceptions
are then permitted when the residual decrease is sufficient.

\begin{figure}[htbp]
\centering
\includegraphics[width=0.78\textwidth]{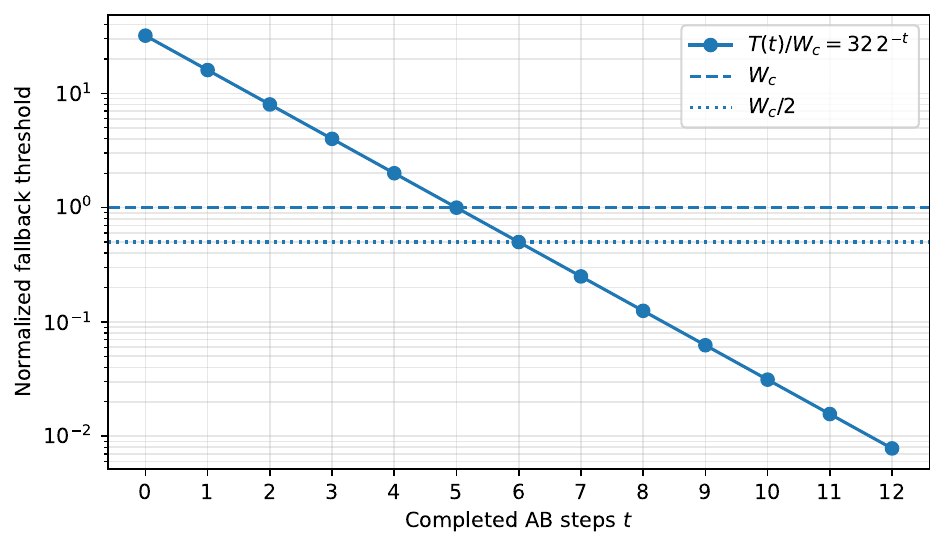}
\caption{Normalized geometric upper bound $32/2^t$, which is also the
threshold for the fixed-delay version. The adaptive threshold
$T_t/S=4/2^t$ is not plotted. An active adaptive phase remains below
the geometric upper bound.}
\label{fig:threshold}
\end{figure}

\begin{corollary}[Finite termination for an absolute width tolerance]
\label{cor:finite_termination}
Let $\tau>0$, and suppose the method stops as soon as $W_n\leq\tau$.  Define
\[
    K=
    \max\left\{
       0,\,
       \left\lceil\log_2\frac{W_0}{\tau}\right\rceil
    \right\}.
\]
Unless an exact zero is found earlier, $W_{qK}\leq\tau$. Thus the
method terminates after at most
\begin{equation}
    N_{\max}=qK
    \label{eq:finite_bound}
\end{equation}
bracket updates.  For $C=32$, one may take $q=7$.
\end{corollary}

\begin{proof}
If an exact zero is found earlier, the conclusion is immediate.
Otherwise, by the definition of $K$, $W_0/2^K\leq\tau$. Applying
Theorem~\ref{thm:quantitative} at $n=qK$ gives
\[
    W_{qK}
    \leq
    W_0\,2^{-\lfloor qK/q\rfloor}
    =
    \frac{W_0}{2^K}
    \leq\tau.
\]
\end{proof}

\subsection{An error bound at termination}

Let
\[
    Z=\{z\in[a_0,b_0]:f(z)=0\}
\]
be the set of all zeros in the initial bracket.  If the algorithm returns a
point $x_n\in I_n$ after a width test $W_n\leq\tau$, then
\begin{equation}
    \operatorname{dist}(x_n,Z)\leq\tau.
    \label{eq:distance_zero_set}
\end{equation}
Here $\operatorname{dist}(x_n,Z)$ means the distance from $x_n$ to the nearest
zero in $Z$.  The estimate is immediate: the current bracket contains at least
one zero $z$, and both $x_n$ and $z$ lie in an interval of width at most
$\tau$.

Using the whole zero set rather than a preselected zero is standard in
variational analysis \cite{RockafellarWets1998,DontchevRockafellar2014}.  It is
also natural for periodic functions or for a bracket containing several zeros.

A small residual $|f(x_n)|$ alone is not a scale-independent error bound.
If, however, the bracket contains a unique zero $x^\ast$, $f$ is
differentiable, and
\[
    |f'(x)|\geq m>0
\]
throughout the bracket, the mean value theorem gives
\begin{equation}
    |x_n-x^\ast|
    \leq
    \frac{|f(x_n)|}{m}.
    \label{eq:residual_bound}
\end{equation}

\section{Implementation and Reliable Stopping}
\label{sec:implementation}

In the C\# implementation, the endpoint coordinates are stored in
\texttt{p1.X} and \texttt{p2.X}.  The fields \texttt{y1} and \texttt{y2}
retain the actual function values $\alpha$ and $\beta$, whereas \texttt{p1.Y} and
\texttt{p2.Y} hold the modified ordinates $u$ and $v$.  This separation is
the main correction.  A change to $u$ or $v$ changes the next interpolation
point; it must not change which signs define the bracket.

\subsection{Correspondence between the method and its implementation}
\label{subsec:correction_map}

The two implementation variants number their corrections separately.
Table~\ref{tab:eight_corrections} brings together the six changes in
\texttt{ModAB3C} and the two additional structural changes in
\texttt{ModAB3}. The third numbered comment in the simplified method
concerns the common secant safeguard and is included in row~2. These
eight changes describe the relation to the earlier implementation;
they are not eight separate modifications of the convergence theorem.
Row~6 concerns the fixed-delay version. The adaptive residual rule
replaces that delay, while retaining the separation of actual and
modified ordinates and the same sign-preserving interval update.

\begin{table}[H]
\caption{Eight ModAB changes and their purpose. C1--C6 identify the
numbered comments in \texttt{ModAB3C}; S1--S2 identify those in
\texttt{ModAB3}.}
\label{tab:eight_corrections}
\centering
\small
\begin{tabularx}{\textwidth}{@{}r l X X@{}}
\toprule
No. & Comment & Change & Purpose \\ \midrule
1 & C1 & Store actual values of the function separately from AB modified ordinates.
& Keep interpolation scaling separate from the actual function values.\\
2 & C2 & Compute a finite trial point within the interval before
forming the mixed tolerance.
& Use a valid approximation in the relative tolerance term.\\
3 & C3 & Reject a new \texttt{NaN} residual before sign decisions.
& Exclude undefined values from sign comparisons.\\
4 & C4 & Replace endpoints using the signs of the actual function value only.
& Preserve the mathematical sign-changing bracket.\\
5 & C5 & Restore modified ordinates from actual function values at fallback.
& Initialize the next phase with the actual endpoint values.\\
6 & C6 & Initialize the dyadic rule from the post-bisection width.
& Retain $C=32$, the five initially automatic comparisons, and the bound with $q=7$.\\
7 & S1 & Return a local function wrapper from \texttt{Initialize}.
& Keep the function being evaluated local to each invocation.\\
8 & S2 & Correct the description of \texttt{side}.
& Document $+1$ as a left-endpoint move and $-1$ as a right-endpoint move.\\
\bottomrule
\end{tabularx}
\end{table}

In row~8 only the comment was wrong: the endpoint-update branches were
already consistent with the stated convention.  The value zero means that
no AB step precedes the current step in that phase.  Correcting this
description doesn't change the iteration sequence or the proof.
Similarly, centralizing the secant safeguards removes a local
\texttt{Clamp} call and duplicate non-finite checks without introducing a
new interpolation formula.

Further corrections address overflow when computing the interval length,
rounding of the midpoint outside an interval with subnormal endpoints,
and a final width check after the iteration limit. These corrections,
described below, were proposed in the subsequent review and are separate
from the eight changes in the table.

\subsection{The iteration in finite precision}

The following procedure specifies the adaptive method together with the
additional finite-precision safeguards proposed in the earlier review.
The switching test always uses the endpoints \emph{before} the current
bisection update.  Its decision takes effect after that update.

\begin{enumerate}
\item Order the finite endpoints $a\leq b$, evaluate their actual function values,
and reject \texttt{NaN} or equal nonzero endpoint signs.  Return an endpoint
if its actual function value is a machine zero.  Set $(u,v)=(\alpha,\beta)$, select
bisection, and clear the record of the last endpoint moved in an AB phase.
The tolerances must be finite and nonnegative.

\item Construct the midpoint in bisection mode or the weighted secant point
in AB mode, using the safeguards below.  Test the current width at this
safeguarded point $x$.  If the width test passes, return $x$.  Otherwise
evaluate the actual function value there, reusing $\alpha$ or $\beta$ if rounding
has made $x$ coincide with an endpoint.  A machine zero terminates the
iteration; \texttt{NaN} gives an invalid-result status.

\item In bisection mode, evaluate \eqref{eq:switch_test} from the actual
endpoint values before replacement and the new midpoint residual.  Non-finite
residuals disable switching.  This decision does not yet change the mode
used for the current update.

\item Before replacing an endpoint in an AB step, save $R_t^-$ from
the actual pre-update endpoint values. For either kind of step, identify
the endpoint whose actual function value has the same nonzero sign as
the new residual. In an AB step, if that endpoint was also replaced on
the preceding AB step, multiply the retained modified ordinate by the
factor in Section~\ref{sec:method}; replace a nonpositive or non-finite
computed factor by $1/2$. In both bisection and AB steps, replace the
identified endpoint and assign the new value of the function both as its
actual value and as its modified ordinate. In an AB phase, record
which endpoint was replaced. The actual function value at the
retained endpoint remains unchanged.

\item After an AB update, first test $W_t\le T_t$. If it passes, reset the exception
counter. Otherwise allow continuation only if
$|f(x_t)|<\eta R_t^-$ and the counter is below $M$, incrementing the counter
in this case. If neither condition permits continuation, select bisection,
restore $(u,v)=(\alpha,\beta)$ and clear the side record. Halve the threshold
only for a continuing phase. After a switching bisection, start a phase
with threshold $BW_c$, counter zero and an empty side record.

\item Repeat up to the chosen update limit.  After the last allowed update,
test the width once more at the safeguarded midpoint.  If this test also
fails, return a maximum-iterations status rather than a successful root.
\end{enumerate}

Sign comparisons and termination at a computed zero use only the actual
function evaluations.
Comparing signs directly avoids forming a residual product that might
overflow or underflow.  The function is evaluated through a wrapper local to each invocation,
but the C\# property \texttt{EvaluationCount} remains a shared counter.
The local callback prevents nested or parallel calls from replacing the
function being evaluated; it does not make that shared count a reliable
per-call result under concurrency.  The revised benchmark therefore counts
evaluations in its own local wrapper.  An endpoint residual reused from
storage is not a new function evaluation.

\subsection{Floating-point safeguards}

Let $\Omega$ be the largest finite binary64 number.  For nonzero finite modified
ordinates set $A=|u|$ and $B=|v|$ and evaluate the secant as
\[
 x=\frac{B}{A+B}a+\frac{A}{A+B}b.
\]
If $\max(A,B)\geq \Omega/4$, first divide both magnitudes by their maximum.
If the magnitudes or denominator are unusable, or the final weighted sum is
non-finite, use the midpoint instead.  Compute that midpoint as
$a/2+b/2$, then project the result onto $[a,b]$.  The secant result is
projected in the same way.  Even the midpoint needs this last check:
separately halving two equal smallest positive subnormal endpoints rounds
both halves to zero.

The switching factor can be evaluated from the magnitudes of the actual
endpoint values as
\[
 k=\left(1-\frac{\bigl||\beta|-|\alpha|\bigr|}
                    {2(|\alpha|+|\beta|)}\right)^2.
\]
Normalize these magnitudes by their maximum when it is at least $\Omega/4$.
Form the chord midpoint as $\alpha/2+\beta/2$.  When either magnitude in
\eqref{eq:switch_test} is at least $\Omega/4$, divide the midpoint residual and
chord value by their common maximum before comparing the two sides.
The comparison remains strict.

After an AB multiplication, a modified ordinate that underflows to zero is
replaced by the smallest positive subnormal number with its previous sign;
an infinite product is replaced by the corresponding signed value $\Omega$.
These replacements affect only the interpolation ordinates. They do not
correct errors in function evaluation or establish the sign of the
underlying real-valued function.

The pinned \texttt{ModABCorr} snapshot used for the controlled numerical
comparison retains the direct midpoint $a/2+b/2$, the direct width comparison,
and the update-limit return without the additional final-width check. The
projection of a subnormal midpoint and the overflow-aware width comparison
described here are further safeguards from the earlier review, not changes
silently inserted into that experiment. The reported ordinary-scale cases do
not validate these extreme-value paths.

The two earlier stored variants differ when an AB step returns an infinite true
residual.  The commented \texttt{ModAB3C} skips the AB rescaling for that
step, updates the bracket using signs of the actual function values, restores the modified ordinates,
and forces the next step to be bisection.  The simplified \texttt{ModAB3}
may retain that value while its AB rule remains active; an unusable
secant then produces a midpoint through the shared helper.  The two variants therefore need not give bitwise identical results in
these exceptional cases.  Infinite
actual function values lie outside Assumption~1, so this difference does not change
the exact-arithmetic convergence result.

\begin{revision}
\subsection{Evaluation of the switching expression}
\label{subsec:simplified_arithmetic}

A later implementation snapshot, commit \texttt{51a6d1b}
\cite{ModABSimplifiedSnapshot2026}, evaluates the right-hand side of
\eqref{eq:switch_test} in distributed form,
\[
 k|y_m|+k|y_3|,
\]
and retains a finite-value check on $f_2-f_1$ before attempting the switch.
Opposite signs prevent overflow in $f_1+f_2$ for finite inputs, but not in
$f_2-f_1$. If the latter is non-finite, the implementation remains in
bisection mode. Lemma~\ref{lem:switch_cubic} bounds the exact distributed
expression by $H=\max(|f_2-f_1|,|y_3|)$. The rounding-error estimates
in Section~\ref{subsec:switch_rounding} distinguish the accuracy of this
distributed evaluation from a uniform guarantee for a scaled variant.

In floating-point arithmetic, the equivalent expression
$k(|y_m|+|y_3|)$ can behave differently: the sum in parentheses may
overflow before multiplication by $k$ reduces its magnitude.
With $\Omega=\texttt{double.MaxValue}$ and
\[
 f_1=-\Omega/4,\qquad f_2=3\Omega/4,\qquad y_3=\Omega,
\]
the C\# check gives approximately $1.2640029854500657\cdot10^{308}$
for the distributed expression and \texttt{Infinity} for the grouped one.
These results are recorded with the separate snapshot experiment in
Section~\ref{subsec:native_cap}.

The same distinction between exact algebra and evaluation order applies
to a weighted secant. A convex-combination formula is bounded in exact
arithmetic, but its independently rounded weights need not sum to exactly
one. In the later snapshot, projection onto the finite bracket also handles
an infinite weighted sum. For example, magnitudes $A=1$, $B=2^{-53}$ and
endpoints $a=3\Omega/4$, $b=\Omega$ give an infinite raw weighted sum in
the C\# check, and projection returns $b$. This calculation illustrates the role of projection in that particular
case but it does not establish correctness for all exceptional inputs.

The arithmetic simplifications concern the evaluation of each step,
whereas the bound on consecutive exceptions concerns convergence of the
intervals. The minimal correction examined below keeps the simplified
arithmetic and restores only that bound and the strict residual
inequality. The additional midpoint and interval-width safeguards
are separate.
\end{revision}

\begin{revision}
\subsection{Rounding errors in the switching criterion}
\label{subsec:switch_rounding}

The cubic estimate concerns the exact value of the switching expression.
To study its evaluation, we regard the supplied binary64 numbers
$f_1,f_2,f_3$ as exact input data, with $f_1f_2<0$ and all three values
finite. Here $f_3$ denotes the computed midpoint residual. Errors already
present in these function evaluations are not included in the following
bounds. Thus the comparison to be reproduced is
\begin{equation}
 L<R,\qquad
 L=|f_3-y_m|,\qquad R=k(|y_m|+|f_3|),
 \label{eq:switch_exact_sides}
\end{equation}
where $D,y_m,v,k$ are defined by \eqref{eq:cubic_definitions}.

We assume binary64 arithmetic with rounding to nearest, ties to even,
and gradual underflow. Write $\operatorname{fl}$ for a single correctly
rounded operation, and put
\[
 \epsilon=2^{-53},\qquad
 \omega=2^{-1074},\qquad
 \gamma_j=\frac{j\epsilon}{1-j\epsilon}.
\]
The indicated order of operations is part of the analysis and intermediate
results are rounded to binary64, without reassociation or replacement
of separate operations by fused operations. These are the usual
rounding assumptions \cite{Goldberg1991,Higham2002,IEEE7542019}.

\paragraph{The unscaled distributed expression.}
Consider the evaluation
\begin{align}
 \widehat D&=|\operatorname{fl}(f_2-f_1)|,&
 \widehat y&=\operatorname{fl}\bigl(\operatorname{fl}(f_1+f_2)/2\bigr),
 \notag\\
 \widehat v&=\operatorname{fl}(|\widehat y|/\widehat D),&
 \widehat r&=\operatorname{fl}(1-\widehat v),&
 \widehat k&=\operatorname{fl}(\widehat r^2),
 \notag\\
 \widehat L&=|\operatorname{fl}(f_3-\widehat y)|,&
 \widehat R&=\operatorname{fl}\bigl(
   \operatorname{fl}(\widehat k|\widehat y|)
   +\operatorname{fl}(\widehat k|f_3|)\bigr).
 \label{eq:switch_raw_evaluation}
\end{align}
As in the implementation described above, a non-finite $\widehat D$
precludes switching.

\begin{proposition}[Error bound when the relative rounding model applies]
\label{prop:switch_raw_error}
Suppose no operation in \eqref{eq:switch_raw_evaluation} overflows,
and every nonzero exact intermediate result lies in the normal range.
With $H=\max(D,|f_3|)$,
\begin{equation}
 \bigl|(\widehat R-\widehat L)-(R-L)\bigr|
 \le 16\epsilon H.
 \label{eq:switch_raw_error}
\end{equation}
The subtraction $\widehat R-\widehat L$ in this bound is an exact
subtraction of the two computed numbers. In particular, the comparison
$\widehat L<\widehat R$ agrees with $L<R$ whenever
$|R-L|>16\epsilon H$.
\end{proposition}

\begin{proof}
Under the stated assumptions, division of the rounded sum by two is
exact, and
\[
 |\widehat y-y_m|\le\epsilon|y_m|,\qquad
 |\widehat v-v|\le\gamma_3v\le\gamma_3/2.
\]
The quotient estimate accounts for the rounded endpoint sum, endpoint
difference and division. It follows by bounding the relative factor
$(1+\delta_1)(1+\delta_3)/(1+\delta_2)$, where
$|\delta_i|\le\epsilon$. If $y_m=0$, the corresponding quantities are
zero and the estimate remains valid.

The subtraction defining $\widehat r$ and the subsequent square give
\[
 \bigl|\widehat k-(1-\widehat v)^2\bigr|\le\gamma_3.
\]
Indeed, their relative factor contains the rounding error of the
subtraction twice and that of the multiplication once. Since
$0\le v\le1/2$ and $0\le\widehat v<1$, both $1-v$ and
$1-\widehat v$ belong to $[0,1]$. Therefore
\[
 |\widehat k-k|
 \le\gamma_3+2|\widehat v-v|
 \le2\gamma_3.
\]
Also $0\le\widehat k\le1$. For the left-hand side,
\[
 |\widehat L-L|
 \le\epsilon|f_3-\widehat y|+|\widehat y-y_m|
 \le\left(2\epsilon+\frac{\epsilon^2}{2}\right)H.
\]
For the right-hand side, the two rounded products and their rounded
sum differ from $\widehat k(|\widehat y|+|f_3|)$ by at most
$\gamma_2\widehat k(|\widehat y|+|f_3|)$. Hence
\[
 |\widehat R-R|
 \le\left[
   \frac{3+\epsilon}{2}\gamma_2
   +3\gamma_3+\frac{\epsilon}{2}
 \right]H.
\]
Adding the last two estimates gives
\[
 \bigl|(\widehat R-\widehat L)-(R-L)\bigr|
 \le\left[
   \frac{3+\epsilon}{2}\gamma_2+3\gamma_3
   +\frac{5\epsilon+\epsilon^2}{2}
 \right]H
 <16\epsilon H.
\]
This proves the bound and the assertion about the strict comparison.
\end{proof}

The normal-range assumption matters. For example, let
$f_1=-\omega$, $f_2=2\omega$ and $f_3=\omega$.
In exact arithmetic,
\[
 y_m=\omega/2,\qquad v=1/6,\qquad k=25/36,\qquad
 R-L=13\omega/24>0.
\]
In \eqref{eq:switch_raw_evaluation}, however, $\widehat y=0$ because
the tie $\omega/2$ rounds to zero. Consequently $\widehat k=1$ and
$\widehat L=\widehat R=\omega$: the computed comparison declines a
switch that the exact comparison permits. The exact cubic bound
therefore does not imply agreement of the two decisions. Nor does
Proposition~\ref{prop:switch_raw_error}, which assumes that overflow
has not occurred, prove its absence at the largest input magnitudes.

\paragraph{A scaled evaluation valid throughout the finite range.}
Both restrictions can be avoided by evaluating the symmetry factor
and the two sides with separate scales. Let
\begin{equation}
 \rho=\frac{\min(|f_1|,|f_2|)}{\max(|f_1|,|f_2|)},\qquad
 s=\max(|f_1|,|f_2|,|f_3|),\qquad
 p_i=\frac{f_i}{s},\qquad h=\frac{p_1+p_2}{2}.
 \label{eq:switch_scaled_data}
\end{equation}
These definitions use exact real arithmetic. Since the endpoint
values have opposite signs,
\[
 v=\frac{1-\rho}{2(1+\rho)},\qquad
 k=\kappa(\rho):=
 \left(\frac{1+3\rho}{2(1+\rho)}\right)^2.
\]
The exact normalized difference is
\begin{equation}
 G=k(|h|+|p_3|)-|p_3-h|=\frac{R-L}{s}.
 \label{eq:switch_scaled_gap}
\end{equation}
Thus $G>0$ is precisely the original switching condition.

Compute the factor in the following order:
\begin{align}
 \widehat\rho
 &=\operatorname{fl}\left(
   \frac{\min(|f_1|,|f_2|)}{\max(|f_1|,|f_2|)}
   \right),\notag\\
 \widehat a&=\operatorname{fl}(1-\widehat\rho),&
 \widehat b&=\operatorname{fl}(1+\widehat\rho),\notag\\
 \widetilde v&=\operatorname{fl}\left(
    \frac{\widehat a/2}{\widehat b}\right),&
 \widetilde r&=\operatorname{fl}(1-\widetilde v),&
 \widetilde k&=\operatorname{fl}(\widetilde r^2).
 \label{eq:switch_scaled_factor}
\end{align}
Next compute
\begin{align}
 \widehat p_i&=\operatorname{fl}(f_i/s),&
 \widehat h&=\operatorname{fl}\bigl(
        \operatorname{fl}(\widehat p_1+\widehat p_2)/2\bigr),
 \notag\\
 \widetilde L&=|\operatorname{fl}(\widehat p_3-\widehat h)|,&
 \widetilde R&=\operatorname{fl}\bigl(
       \operatorname{fl}(\widetilde k|\widehat h|)
       +\operatorname{fl}(\widetilde k|\widehat p_3|)\bigr),
 \notag\\
 \widehat G&=\operatorname{fl}(\widetilde R-\widetilde L).
 \label{eq:switch_scaled_evaluation}
\end{align}
The maximum defining $s$ selects an input magnitude and requires no
rounding. The endpoint scale in \eqref{eq:switch_scaled_factor} prevents
the denominator of the symmetry ratio from vanishing when $|f_3|$ is
much larger than both endpoint magnitudes.

\begin{theorem}[A uniform rounding bound for the scaled comparison]
\label{thm:switch_scaled_error}
For all finite binary64 inputs with $f_1f_2<0$, the operations in
\eqref{eq:switch_scaled_factor}--\eqref{eq:switch_scaled_evaluation}
produce finite values and have no zero divisor. Under the rounding
assumptions stated above, including gradual underflow,
\begin{equation}
 |\widetilde k-k|\le8\epsilon,\qquad
 |\widehat h-h|\le3\epsilon,\qquad
 |\widehat G-G|\le25\epsilon.
 \label{eq:switch_scaled_error}
\end{equation}
No lower bound on the nonzero input magnitudes is necessary.
\end{theorem}

\begin{proof}
First observe that rounding any real number in $[-2,2]$ to binary64
introduces absolute error at most $\epsilon$. The largest spacing
below $2$ is $2^{-52}=2\epsilon$; subnormal spacing is smaller.
This absolute bound therefore remains valid when relative error
estimates fail near zero.

The computed ratio satisfies $0\le\widehat\rho\le1$.
Consequently,
\[
 0\le\widehat a\le1,\qquad
 1\le\widehat b\le2,\qquad
 0\le\widetilde v\le1/2,\qquad
 1/2\le\widetilde r\le1,\qquad
 1/4\le\widetilde k\le1.
\]
Division of $\widehat a$ by two is exact: either $\widehat a=0$,
or $\widehat\rho<1$ and $\widehat a\ge2^{-53}$.
Thus all operations defining the factor are finite and their
divisors are positive.

To bound its error, note that
\[
 \kappa'(\rho)=\frac{1+3\rho}{(1+\rho)^3},\qquad
 0<\kappa'(\rho)\le1\quad(0\le\rho\le1).
\]
Since $|\widehat\rho-\rho|\le\epsilon$, replacing $\rho$ by
$\widehat\rho$ changes $\kappa$ by at most $\epsilon$.
For this fixed $\widehat\rho$, put
$a=1-\widehat\rho$, $b=1+\widehat\rho$ and $v_0=a/(2b)$.
The rounding bounds give
$|\widehat a-a|\le\epsilon$ and
$|\widehat b-b|\le\epsilon$, while $b,\widehat b\ge1$. Hence
\[
 \left|\frac{\widehat a}{2\widehat b}-\frac{a}{2b}\right|
 \le\frac{|\widehat a-a|}{2\widehat b}
   +\frac{a|\widehat b-b|}{2b\widehat b}
 \le\epsilon.
\]
Rounding the division adds at most $\epsilon$, so
$|\widetilde v-v_0|\le2\epsilon$. Rounding the next subtraction
gives $|\widetilde r-(1-v_0)|\le3\epsilon$. Both quantities being
squared lie in $[0,1]$, and therefore
\[
 |\widetilde k-\kappa(\widehat\rho)|
 \le 2(3\epsilon)+\epsilon=7\epsilon.
\]
Combining this with the perturbation of $\rho$ proves the first bound.

Each normalized input satisfies $|p_i|\le1$ and
$|\widehat p_i-p_i|\le\epsilon$. Rounding preserves the opposite
endpoint signs, except that a sufficiently small normalized endpoint
may become zero. In either case,
$|\widehat p_1+\widehat p_2|\le1$ and
$|\widehat h|\le1/2$. The endpoint divisions, the addition and the
halving therefore give
\[
 |\widehat h-h|
 \le\frac{\epsilon+\epsilon+\epsilon}{2}+\epsilon
 \le3\epsilon.
\]
The last $\epsilon$ also covers a subnormal halving.

Write $L_s=|p_3-h|$ and $R_s=k(|h|+|p_3|)$.
The subtraction for the left-hand side gives
\[
 |\widetilde L-L_s|
 \le\epsilon+3\epsilon+\epsilon=5\epsilon.
\]
For the two products on the right-hand side, use
$|p_3|\le1$, $|h|\le1/2$ and $\widetilde k\le1$:
\begin{align*}
 \left|\operatorname{fl}(\widetilde k|\widehat p_3|)
              -k|p_3|\right|
 &\le\epsilon+\epsilon+8\epsilon=10\epsilon,\\
 \left|\operatorname{fl}(\widetilde k|\widehat h|)
              -k|h|\right|
 &\le\epsilon+3\epsilon+4\epsilon=8\epsilon.
\end{align*}
The rounded sum contributes at most another $\epsilon$, giving
$|\widetilde R-R_s|\le19\epsilon$.
All these exact intermediate results have magnitude at most $2$:
the products are bounded by $1$ and $1/2$, respectively, and
both sides of the comparison are at most $3/2$. Thus the absolute
rounding bound applies throughout, and no overflow is possible.
Finally, rounding their difference contributes at most $\epsilon$:
\[
 |\widehat G-G|
 \le19\epsilon+5\epsilon+\epsilon=25\epsilon.
\]
\end{proof}

\begin{corollary}[A sufficient condition for a reliable switching decision]
\label{cor:switch_guard}
With the scaled evaluation above,
\begin{equation}
 \widehat G>32\epsilon=2^{-48}
 \quad\Longrightarrow\quad L<R.
 \label{eq:switch_guard}
\end{equation}
Likewise, $\widehat G<-32\epsilon$ implies $L>R$.
If neither inequality holds, retaining bisection avoids an uncertified
switch. Every input with $G>57\epsilon$ is accepted by
\eqref{eq:switch_guard}, and every input with $G<-57\epsilon$
is certified as not satisfying the exact switching inequality.
\end{corollary}

\begin{proof}
Theorem~\ref{thm:switch_scaled_error} gives
$G\ge\widehat G-25\epsilon>7\epsilon>0$ in the first case,
and the analogous negative bound in the second. Conversely,
$G>57\epsilon$ implies $\widehat G>32\epsilon$, and similarly
for the negative case.
\end{proof}

The number $2^{-48}\simeq3.55\cdot10^{-15}$ is a dimensionless
rounding margin, not a residual tolerance or a stopping tolerance for
the root. The constants are conservative bounds, not estimates of
typical error. In the subnormal example above, the exact scaled
difference is $G=13/48$, so the scaled rule accepts the switch.
Near equality, however, a strict comparison can change under rounding
even when every intermediate result is normal.

A separate Python reference evaluation checked 23,816 input triples
against exact rational arithmetic, including extreme exponents,
subnormal values, points inside the switching region, and neighboring
binary64 numbers at its exact boundary. All five bounds used in the
proof, with constants $8$, $3$, $5$, $19$ and $25$, held in these
checks, and every decision certified by the guard agreed with the
exact comparison. The same calculation verified
\eqref{eq:switch_raw_error} on the 22,615 triples satisfying its
normal-range assumptions. These are finite checks of the formulas
and their reference implementation; the universal claims follow from
the proofs above. The concentration of cases near the boundary is
deliberate and gives no estimate of how often rounding changes a
decision in a practical root-finding problem.

The scaled rule is an additional arithmetic variant and it has not been
inserted retrospectively into the C\# benchmark results reported in
this paper. It preserves the mathematical convergence argument because
this argument does not require a particular switching criterion.
It does not certify the signs or accuracy of the supplied function
values, and it does not extend the exact-arithmetic convergence theorem
to arbitrary floating-point function evaluations.
\end{revision}

\subsection{Machine zeros and interval-width termination}

In binary64 arithmetic, evaluation of the supplied expression defines a computed function
$\widehat F$.  Its exact zero need not be a zero of the real expression
$f$, and a real zero need not evaluate to machine zero.  This distinction
is standard in floating-point error analysis
\cite{Goldberg1991,Higham2002,IEEE7542019}.  The two stopping conditions are
\begin{equation}
 \widehat F(x)=0
 \qquad\text{or}\qquad
 b-a\leq\varepsilon_{\rm abs}+\varepsilon_{\rm rel}|x|.
 \label{eq:two_exits}
\end{equation}
The candidate must already be finite and inside $[a,b]$ before its
magnitude enters the relative tolerance.

If the directly computed width or tolerance overflows, compare instead
\[
 \frac b2-\frac a2
 \ \leq\
 \frac{\varepsilon_{\rm abs}}2+
       \varepsilon_{\rm rel}\frac{|x|}{2},
\]
requiring a finite left-hand side.  This prevents an erroneous acceptance
based on $\infty\leq\infty$.  It is still an ordinary rounded comparison,
not an interval-arithmetic certificate.

The geometric conclusion \eqref{eq:distance_zero_set} requires the computed
endpoint signs to agree with those of a continuous real function, and
requires the width comparison not to round an insufficient width into an
accepted one.  A formally verified implementation would need validated
function enclosures and conservative rounding. We do not claim these properties here.

\subsection{Changes to the comparison methods and compilation corrections}
\label{subsec:related_changes}

We also corrected the implementation of Ridders' method that we use for comparison.
It uses the local callback, recognizes endpoint machine zeros, and rejects
\texttt{NaN} both at the midpoint and at the subsequent trial point before
either value enters a sign update.  To protect the interpolation, let
$m$ be the midpoint, divide the three residuals by
$S=\max(|f(a)|,|f(b)|,|f(m)|)$, and denote the normalized values by
$\widetilde\alpha,\widetilde\beta,\widetilde\gamma$.  The trial point is
\[
 x_R=m+(m-a)
 \frac{\operatorname{sign}(f(a))\,\widetilde\gamma}
 {\sqrt{\widetilde\gamma^2-\widetilde\alpha\widetilde\beta}}.
\]
If the scale, radicand or final arithmetic is unusable, or the trial point
reaches or leaves an endpoint, the already evaluated midpoint is used
instead.  The stopping test uses the whole bracket width rather than the
distance between successive Ridders points.  Known midpoint and endpoint
residuals are reused, and the sign update retains an ordered bracket.
These comparator changes do not enter the ModAB convergence proof.

In the earlier patch the other solver files were adapted to the new
\texttt{Initialize} signature by receiving the local wrapper through
\texttt{out F}; those call-site edits did not alter their numerical
formulas.  The later build repair in \texttt{ModAB3C} removed a stale
reference to the local \texttt{Clamp} and made the evaluation explicit as
\texttt{F(x3)} before passing the value to \texttt{Node(double,double)}.
This resolves the obsolete-helper and constructor-interface issues without
changing the mathematical update.  

\section{Formal Verification of Switching and Convergence}
\label{sec:formalization}

Theorem~\ref{thm:switch_scaled_error} separates the error in evaluating
the switching expression from any error already present in the supplied
function values. The same distinction is retained in the formal proof.
The inputs are three finite binary64 numbers, interpreted as exact real
numbers, and the assertion concerns the result of the specified sequence
of arithmetic operations. Version~0.4.0 of the formalization covers the scaled
switching criterion and its error estimates, the cubic bound, the bracket
invariants, the phase and global width bounds, exact-arithmetic convergence,
the residual-to-distance estimate, and the two exact AB examples. The formal
global convergence proof follows the two cases in
Section~\ref{sec:convergence}: infinitely many bisections and a final
continuing AB phase. The unscaled rounding-error analysis is also formalized
under its stated normal-range and no-overflow assumptions.

The proof has been checked in Lean~4 using the real-number theory of
mathlib and the floating-point definitions and theorems of
FloatLib~\cite{LeanTheoremProving,FloatLib2026}. Its first task is to
establish the local error bounds from a concrete definition of
rounding. Its second task is to show that the operations on binary64
representations have exactly the real values described by this
rounding map. Both steps are needed: an estimate obtained by propagating
assumed local errors would leave open whether those assumptions hold
for the implemented operations.

\subsection{The rounding map and the two divisions by two}

For $x\ne0$, let $m(x)$ be the unique integer satisfying
\[
 2^{m(x)-1}\le |x|<2^{m(x)},
\]
and put
\begin{equation}
 e(x)=\max\{m(x)-53,-1074\},\qquad
 \operatorname{rn}(x)=2^{e(x)}N\bigl(x2^{-e(x)}\bigr),
 \qquad \operatorname{rn}(0)=0.
 \label{eq:formal_rounding_map}
\end{equation}
Here $N(t)$ is the nearest integer to $t$, with the even integer chosen
when the two distances are equal. The lower bound on $e(x)$ makes the
smallest spacing equal to $\omega=2^{-1074}$. It therefore includes
subnormal numbers and rounding to zero without imposing a relative
error estimate near zero.

The map in \eqref{eq:formal_rounding_map} has no upper exponent bound.
For a value to be a finite binary64 number, it must also satisfy
\begin{equation}
 |x|\le\Omega=(2^{53}-1)2^{971}.
 \label{eq:formal_largest_finite}
\end{equation}
This additional condition is checked when the real-valued map is
related to the finite format. It is not inferred merely from the
precision of the significand.

\begin{lemma}[Absolute rounding error on the required range]
\label{lem:formal_rounding_error}
The map $\operatorname{rn}$ is nondecreasing and fixes every number
representable in its binary grid. Moreover,
\begin{equation}
 |x|\le2\quad\Longrightarrow\quad
 |\operatorname{rn}(x)-x|\le\epsilon,
 \qquad |\operatorname{rn}(x)|\le2.
 \label{eq:formal_local_error}
\end{equation}
Consequently, every rounded value in this statement is finite in
binary64.
\end{lemma}

\begin{proof}
Monotonicity and preservation of representable values are established
for the rounding map in the underlying formal theory. To obtain the
error bound needed here, suppose first that $0<|x|<2$. Then $m(x)\le1$
and hence $e(x)\le-52$. Since $N$ changes its argument by at most
$1/2$,
\[
 |\operatorname{rn}(x)-x|\le\tfrac12\,2^{e(x)}
 \le2^{-53}=\epsilon.
\]
The remaining values $0$, $2$ and $-2$ are representable and are
unchanged by rounding. Monotonicity, applied between the fixed points
$-2$ and $2$, gives the second inequality in
\eqref{eq:formal_local_error}. Finally, $2<\Omega$, so the finite-range
condition is satisfied. Each of these steps is checked in the Lean
proof; the subnormal range is covered by the same absolute estimate.
\end{proof}

There are two divisions by two in
\eqref{eq:switch_scaled_factor}--\eqref{eq:switch_scaled_evaluation},
and their treatment is different. The division of $\widehat a$ by two
is exact. The division used to form $\widehat h$ may round, since its
numerator can be as small as a subnormal number. For example,
$\operatorname{rn}(\omega/2)=0$. The formal evaluation retains a
rounding operation at both places and proves that the first one can
be removed.

\begin{lemma}[Exact halving in the symmetry factor]
\label{lem:formal_factor_halving}
If $0\le\rho\le1$, $\widehat\rho=\operatorname{rn}(\rho)$ and
$\widehat a=\operatorname{rn}(1-\widehat\rho)$, then
\begin{equation}
 \operatorname{rn}(\widehat a/2)=\widehat a/2.
 \label{eq:formal_factor_halving}
\end{equation}
\end{lemma}

\begin{proof}
Monotonicity gives $0\le\widehat\rho\le1$. If
$\widehat\rho=1$, then $\widehat a=0$ and the assertion follows.
Otherwise, the spacing immediately below $1$ is $2^{-53}$, so
\[
 1-\widehat\rho\ge2^{-53}=\epsilon.
\]
More explicitly, every representable $x\in[1/2,1)$ is an integer
multiple of $\epsilon$. The strict inequality $x<1=2^{53}\epsilon$
therefore gives $x\le(2^{53}-1)\epsilon$. For $x<1/2$ the same
upper bound is immediate. This proves that there is no representable
number strictly between $1-\epsilon$ and $1$.

Since $\epsilon$ is itself representable, monotonicity now gives
$\widehat a\ge\operatorname{rn}(\epsilon)=\epsilon$.
Also $\widehat a\le1$. Dividing such a representable number by two
lowers its exponent by one and preserves its significand; the result
is at least $2^{-54}$ and remains within the normal range. It is
therefore representable and fixed by rounding. The formal proof
establishes both the spacing assertion and preservation of
representability under this exponent change.
\end{proof}

\subsection{The finite-format statement}

Let $\mathcal B$ denote the set of finite binary64 representations and
let $\nu:\mathcal B\to\mathbb R$ assign to each representation its
exact real value. The two representations of zero have the same value
under $\nu$. Arithmetic on representations is defined with IEEE
rounding to nearest, ties to even, including gradual underflow.

For finite operands $u,w$ and an arithmetic operation $\circ$, the
formal library supplies the relation between the executable operation
and its exact real result. In the range used here, this gives
\begin{equation}
 \nu\bigl(\operatorname{op}_{\circ}(u,w)\bigr)
   =\operatorname{rn}\bigl(\nu(u)\circ\nu(w)\bigr)
 \quad\text{if}\quad
 |\nu(u)\circ\nu(w)|\le2,
 \label{eq:formal_operation_correspondence}
\end{equation}
where division additionally requires $\nu(w)\ne0$.
The result is proved finite as part of the assertion, so $\nu$ is
defined on its left-hand side. This is used separately for addition,
subtraction, multiplication and division. Absolute value and the
selection of a minimum or maximum preserve the corresponding real
values exactly.

Write $\mathcal E_{64}(w_1,w_2,w_3)$ for the binary64 representation
obtained by evaluating the scaled difference in the stated order,
including the final subtraction. Write
$\mathcal E_{\rm rn}(f_1,f_2,f_3)$ for the real-valued expression
obtained by placing $\operatorname{rn}$ at the same arithmetic
operations. In particular, it initially contains
$\operatorname{rn}(\widehat a/2)$; its simplification uses
Lemma~\ref{lem:formal_factor_halving}.

\begin{theorem}[Verified correspondence and switching decision]
\label{thm:formal_switching}
Let $w_1,w_2,w_3\in\mathcal B$, put $f_i=\nu(w_i)$, and suppose
$f_1f_2<0$. Define $\rho,s,p_i,h,G,L,R$ as in
\eqref{eq:switch_exact_sides}--\eqref{eq:switch_scaled_gap}.
Every intermediate arithmetic result in the evaluation of
$\mathcal E_{64}$ is finite, and every divisor is positive. Furthermore,
\begin{equation}
 \widehat G:=\nu\bigl(\mathcal E_{64}(w_1,w_2,w_3)\bigr)
   =\mathcal E_{\rm rn}(f_1,f_2,f_3),\qquad
 |\widehat G-G|\le25\epsilon.
 \label{eq:formal_complete_error}
\end{equation}
The executable comparisons with the binary64 representations of
$2^{-48}$ and $-2^{-48}$ consequently satisfy
\[
 \widehat G>2^{-48}\ \Longrightarrow\ L<R,
 \qquad
 \widehat G<-2^{-48}\ \Longrightarrow\ R<L.
\]
They select the respective branches whenever $G>57\epsilon$ or
$G<-57\epsilon$.
\end{theorem}

\begin{proof}
The hypothesis implies that the endpoint magnitudes and their maximum
are positive. It follows that $s>0$, $0<\rho\le1$ and $|p_i|\le1$.
The normalized endpoint values have opposite signs. Their rounded
values retain the weak sign inequalities even if one of them becomes
zero, and hence $|\widehat p_1+\widehat p_2|\le1$.

Monotonicity of $\operatorname{rn}$ gives the intervals for the
computed factor stated in the proof of
Theorem~\ref{thm:switch_scaled_error}, as well as
$|\widehat h|\le1/2$. The exact results to which rounding is applied
all belong to $[-2,2]$. The endpoint maximum and $s$ are positive,
$\widehat b\ge1$, and the remaining divisor is $2$.
These facts are established before using
\eqref{eq:formal_operation_correspondence}.

Apply that relation successively to the arithmetic subexpressions.
For the initial input expressions, equality of real values is the
definition of $f_i$. If equality has already been established for the
operands of an operation, its range bound and, where needed, its
positive divisor allow \eqref{eq:formal_operation_correspondence} to
be applied. The new representation is finite and its value agrees
with the corresponding rounded real expression. Induction over the
expression therefore proves the equality in
\eqref{eq:formal_complete_error} and finiteness of all intermediate
results.

Lemma~\ref{lem:formal_rounding_error} now supplies each local absolute
error bound, and Lemma~\ref{lem:formal_factor_halving} supplies the
exact first halving. The formal real-arithmetic argument then derives
the bounds with constants $8$, $3$, $5$, $19$ and $25$ used in
Theorem~\ref{thm:switch_scaled_error}. These bounds are conclusions
of the proof, not additional assumptions on the calculation. For
example, the change in $\kappa$ is bounded algebraically by
\[
 \kappa(x)-\kappa(y)
 =(x-y)\frac{1+2x+2y+3xy}{(1+x)^2(1+y)^2},
 \qquad x,y\ge0.
\]
The fraction lies in $[0,1]$, which proves
$|\kappa(x)-\kappa(y)|\le|x-y|$ and supplies the estimate used in
the propagation of errors.

Finally, $2^{-48}=32\epsilon$ is exactly representable. The
comparisons on finite binary64 representations are proved to agree
with the strict order of their real values. Applying the error bound
gives the implications and the $57\epsilon$ separation conditions,
as in Corollary~\ref{cor:switch_guard}. The identity
$G=(R-L)/s$, with $s>0$, returns the assertion to the original
unscaled inequality.
\end{proof}

\subsection{Correspondence with the C\# expression and reproducibility}

The accompanying C\# implementation of the scaled comparison uses
separate addition, subtraction, multiplication and division calls.
A small source translator records their arithmetic expressions in
Lean, preserving the order and nesting of all 17 operations. Equality
with the rounded expression used in the proof is then checked in
Lean. In particular, the two products in the right-hand side are
rounded separately before their sum is formed. The decimal constant
used for the comparison is also checked to equal $2^{-48}$ exactly.
The real-value correspondence is sufficient for the switching
criterion since the distinction between positive and negative zero cannot
change the comparison, and none of the divisors can be zero.

Version~0.4.0 of the development uses Lean~4.34.0 and fixed revisions of
mathlib and FloatLib, recorded in the project. Its audit covers 287 lemmas and
theorems and checks their transitive axiom
dependencies. They use only propositional extensionality, the
quotient principle and classical choice from the standard logical
foundation~\cite{LeanTheoremProving}. There are no unfinished proofs
or additional axioms asserting the correctness of rounding. The
source files, fixed dependency revisions, source checksums and
verification records are retained with the reproducibility project
\href{https://github.com/Proektsoft-EOOD/ModAB-Root-Finding/tree/main/MDPI-Algorithms-Data-2/ModAB.Switching.Lean}{\texttt{ModAB.Switching.Lean}},
which is publicly available in the development repository.

The relation to native execution has a definite scope. The source
translator is an inspected component, not a formally verified C\#
compiler. Lean checks the arithmetic expression that it produces and
the executable IEEE model of that expression. Application of the
result to a native C\# execution additionally requires the compiler,
.NET runtime and processor to perform the stated separate binary64
operations, with the rounding and gradual-underflow assumptions of
Section~\ref{subsec:switch_rounding}. Their complete implementations
are outside the formal proof.

A separate C\# verification checks these arithmetic assumptions on the
recorded execution environment, .NET~10.0.12 on Linux x64. For the
23,816 input triples described above, all 404,872 arithmetic
operations were compared with independent rounding based on exact
integer quotients, remainders and parity. A further 14,268 operations
exercise boundary cases, including signed zeros, subnormal numbers,
normal-range boundaries, ties and overflow. All 419,140 comparisons
agree in all 64 bits. These checks support the correspondence for
the tested environment; the universal error bound follows from the
formal proof.

The verified assertion concerns admissible inputs to the scaled
comparison. The C\# rejection of invalid inputs is covered by
separate tests. Accuracy of the supplied values of the function, the entire
floating-point root-finding implementation and the native compiler are not thereby
verified. The module \texttt{Unscaled.lean} proves
Proposition~\ref{prop:switch_raw_error} by propagating errors through the
specified sequence of concrete rounding operations under its stated
normal-range and no-overflow assumptions. This rounded-real proof does not
establish the native C\# execution or a separate finite-format correspondence
for every operation in the unscaled expression. The scaled comparison remains the separate arithmetic
variant specified above, and its verification does not alter the
benchmark results for the other implementations.

\section{Numerical Experiments and Reproducibility}
\label{sec:numerics}

\subsection{Independent assessment of the root approximations}

Consider $f_p(x)=x^p-10^{-3}$ on $[-10,10]$, with $p$ odd.  For
$p=13,15,17$, the positive roots are approximately $0.588$, $0.631$, and
$0.666$.  Yet a false-position trial point can move so little near an endpoint
that a step-based test stops far from all three roots.  In this case, the small difference between successive approximations
is caused by stagnation. The length of the interval remains large because the
other endpoint is fixed, so a stopping condition based on that length
does not give the same premature termination.

A residual test has a different limitation.  Multiplying the equation by a
small positive constant leaves its roots unchanged but makes every residual
smaller.  In floating-point arithmetic the rescaled expression may even
underflow to zero away from a mathematical root.  For example, a callback
evaluating $10^{-320}(x-\tfrac14)^3$ can return zero at $x=0.26$, although the
horizontal error is $0.01$.  A separate accuracy check based on the same computed zero would inherit
this error and would provide no independent information about the
distance to the root.

For a supplied reference set $Z_i$ within the initial interval, the revised
benchmark checks
\begin{equation}
 e_{x,i}=\min_{z\in Z_i}|\widehat x_i-z|
 \leq \varepsilon_{\rm abs}+\varepsilon_{\rm rel}|\widehat x_i|.
 \label{eq:external_validation}
\end{equation}
It rejects non-finite and out-of-interval answers and does not accept
$\widehat F_i(\widehat x_i)=0$ as a substitute for this test.
The benchmark call passes both requested tolerances explicitly. The proposed
\texttt{IsRoot(x,aTol,rTol)} retains an absolute-only default for the optional
third argument, so legacy two-argument calls are not silently reinterpreted.
It checks finite nonnegative tolerances, finite ordered bounds and candidates,
and usable reference roots inside the initial interval. A null or empty
reference list cannot validate a result. Non-finite tolerance arithmetic is
conservatively rejected; the validator does not evaluate the callback.
The reference roots are numerical approximations, not interval enclosures;
their accuracy and completeness must therefore be audited before interpreting
the resulting classifications as more than reference-based validation.
In particular, a benchmark must not reject a different valid root just
because its reference list is incomplete.

\subsection{Function evaluation counts and classification of test problems}

The validation and counting rules in this subsection are a further
benchmark revision.  In the earlier implementation snapshot,
\texttt{Problem.IsRoot} could still accept a callback zero after failing
its reference-distance test.  That fallback is excluded here; the
benchmark corrections should not be confused with the eight ModAB changes
in Table~\ref{tab:eight_corrections}.

Actual callback evaluations are counted by a wrapper local to each benchmark
invocation.  Solver return status is preserved separately from the
reference-based validation status.  A failed validation does not overwrite
the measured count.  The output includes both the raw count and an explicitly
labelled penalized cost; the common failure penalty is one larger than the
largest observed raw count among the reference-eligible cases in that run.  Thus a method that terminates after only two or three evaluations with
an inaccurate approximation is not ranked ahead of a successful method.
Successful-run summaries and penalized-cost summaries are reported separately.

The broader development benchmark contains 100 cases.  Some are deliberately
discontinuous or have singular endpoints.  Such stress cases can expose
implementation behavior, but they are not tests of the convergence theorem
under Assumption~1.  In particular, the sawtooth cases have jumps through zero
that are not roots; a narrow interval straddling such a jump is not a
root-containing bracket.  These cases are classified separately from continuous sign-changing problems.
The controlled comparison below uses the pinned benchmark definitions.
Separate explicit examples document the proof mechanism and the machine-zero
failures, so these particular checks do not require the whole collection.

\subsection{C\# comparison of the rules for returning to bisection}
\label{subsec:native_adaptive}

The comparison uses the fixed repository version at commit
\texttt{9111cdcc\allowbreak{}6912401e\allowbreak{}b684a7b7\allowbreak{}71df3a46\allowbreak{}9795d5b8}
\cite{ModABSnapshot2026}. We compiled Release with .NET SDK 10.0.401,
runtime 10.0.12, on Ubuntu 24.04 x64, using binary64 values,
$\varepsilon_{\rm abs}=\varepsilon_{\rm rel}=10^{-14}$ and an update limit
of 200. Three variants share the same \texttt{ModABCorr} arithmetic and
differ only in the phase-entry threshold and exception cap: fixed delay
$(B,M)=(16,0)$, early fallback $(2,0)$, and adaptive fallback $(2,3)$ with
$\eta=1/2$. Two further runs use the unmodified repository \texttt{ModAB}
and \texttt{ModABCorr}, giving 500 solves over the 100 named cases.
The instrumented adaptive variant agrees with the unmodified
\texttt{ModABCorr} in returned result and callback count on all 100 cases.
Function evaluations are counted by a local wrapper; validation does not
call that wrapper or the function.

\begin{table}[htbp]
\centering
\caption{Function evaluations in the controlled C\# comparison. Initial
endpoint evaluations are included. Fixed and early variants change only the
adaptive rule's parameters.}
\label{tab:adaptive_counts}
\begin{tabular}{@{}lrrr@{}}\toprule
Case or aggregate & Fixed delay & Early fallback & Adaptive\\\midrule
f64 &9&14&9\\
f84 &9&16&9\\
f92 &24&21&21\\
f93 &23&21&21\\
f94 &24&21&21\\
f95 &22&19&19\\
f96 &23&20&20\\
f97 &24&21&21\\\midrule
Raw sum, 92 reference-eligible cases &1643&1658&1626\\
Raw sum, common 89 validated cases &1590&1605&1573\\
Raw sum, 86 validated non-stress cases &1492&1507&1475\\
Penalized cost, 92 cases &1758&1773&1741\\\bottomrule
\end{tabular}
\end{table}

The eight cases without an eligible reference root are excluded from these
aggregates. The validator rejects f27, f29 and f75 for all three variants;
their 53 measured evaluations in total are retained in the raw 92-case sum,
not counted as successful work. The same 89 cases pass for all variants.
Removing the discontinuous sawtooth cases f82 and f83 and the singular-endpoint
case f100 leaves the 86-case non-stress subtotal. This classification does not
replace a full audit of all remaining function and reference definitions.
The largest observed raw count on the 92 cases is 55, so the separate failure
penalty is 56. It is a scoring convention, not an evaluation count.

Relative to the fixed-delay version, the adaptive rule reduces the total
by 17 function evaluations, all on f92--f97, without increasing the count
on any case in this comparison. It also avoids the increased counts
produced by early return to bisection on f64 and f84. The original repository
\texttt{ModAB} is a different baseline: its raw sum is 1626 in this run, equal
to the adaptive total, and it takes eight evaluations on f84 rather than nine.
Neither the controlled result nor the aggregate supports a universal ranking
of the methods. No wall-clock speedup is claimed.

In the pinned test definitions,
\[
 f_{64}(x)=\tan(x-0.1),\quad [a_0,b_0]=[-1,1],
\]
\[
 f_{84}(x)=\prod_{j=1}^{5}(x-j)-0.05,\quad [a_0,b_0]=[0.5,5.5].
\]
Both observed AB phases approach simple roots. At AB steps 3--5, f64 keeps
width approximately $0.9$ while its residual falls through
$6.11\cdot10^{-5}$, $5.26\cdot10^{-9}$ and $2.78\cdot10^{-17}$.
The next evaluation is a machine zero. For f84 the corresponding width stays
near $2.4875$, while the residual falls through $5.33\cdot10^{-5}$,
$1.20\cdot10^{-8}$ and $7.65\cdot10^{-15}$. The sixth AB update then replaces
the far endpoint, leaving width $2.22\cdot10^{-15}$.

\begin{figure}[htbp]
\centering
\includegraphics[width=0.72\textwidth]{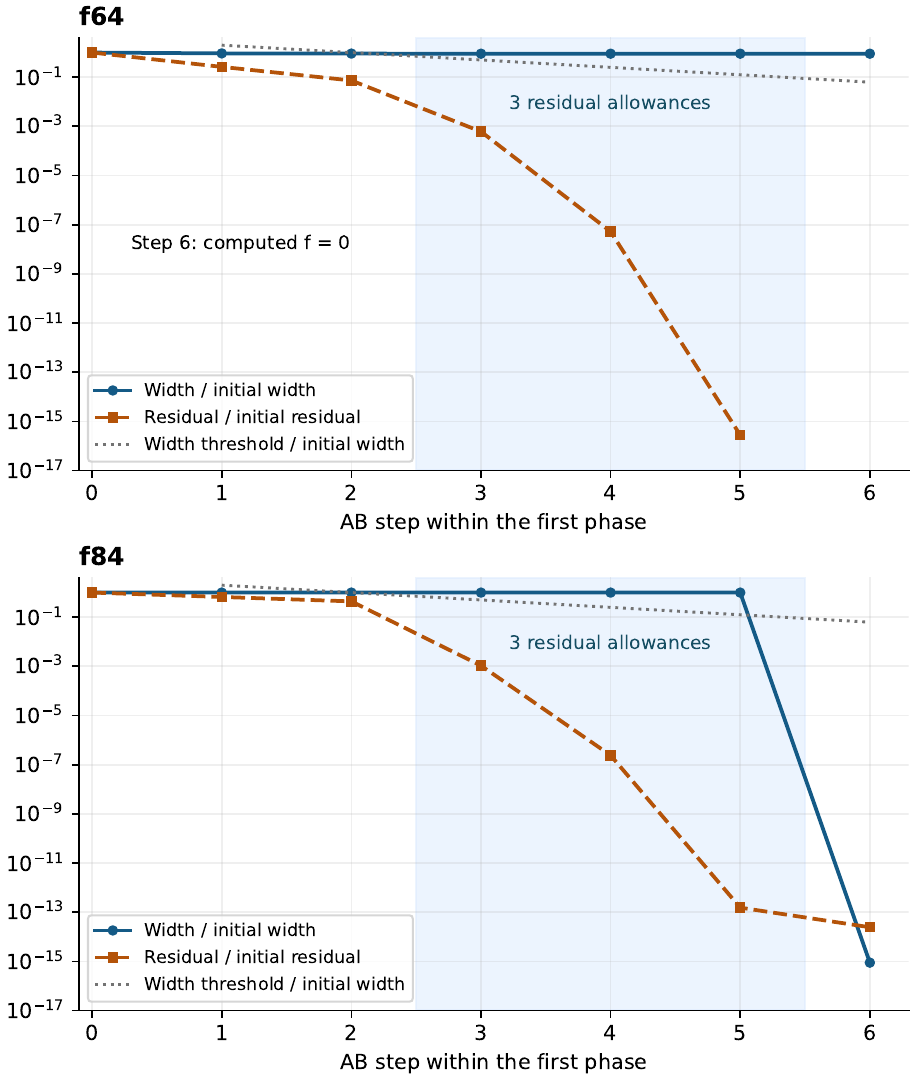}
\caption{First AB phases for f64 and f84. Width and residual are normalized
by their values at phase entry. The shaded region marks three residual
exceptions. The zero at step 6 of f64 is annotated rather than assigned an
artificial positive logarithmic value.}
\label{fig:residual_width}
\end{figure}

The early fallback on f84 also changes the zero being approached. After its
third AB step the bracket is approximately $[3.0124891,5.5]$. The following
midpoint, $4.2562446$, has negative residual, so replacing the left endpoint
excludes the nearby zero $3.01250244\ldots$. The remaining search converges to
$5.00207436\ldots$. The retained interval still contains a zero, as required, but it no
longer contains the zero approached by the preceding AB steps. The odd-power cases are different:
in their initial flat region the smaller endpoint residual magnitude
scarcely decreases, so the
adaptive rule returns to bisection without granting exceptions.

\FloatBarrier

\begin{revision}
\subsection{A numerical example requiring the bound on residual exceptions}
\label{subsec:native_cap}

To examine the effect of the bound on consecutive exceptions in the
implementation, we performed a
separate experiment on commit
\texttt{51a6d1b421fa6b5422b974b61c88acb3ae1945bb}
\cite{ModABSimplifiedSnapshot2026}. This is a later snapshot than the one
used in Section~\ref{subsec:native_adaptive}; the two experiments must not
be interpreted as using a single source revision. The later
\texttt{SgModAB} implementation permits unlimited residual exceptions.
The comparison uses that implementation and a minimal patch restoring
$M=3$, resetting the counter at phase entry and after a passing width
comparison, and accepting a residual exception only under the strict
inequality in \eqref{eq:residual_exception}. Equivalently, a failed width
test returns to bisection when the residual is greater than or equal to
half the pre-update minimum, or when the allowance has been exhausted.
The unpatched comparison uses a strict greater-than rejection and therefore
also admits equality; this difference does not affect the counterexample,
whose residual ratios are strictly below one half.

The recorded run used .NET SDK 10.0.401, Linux x64, Release configuration,
$\varepsilon_{\rm abs}=\varepsilon_{\rm rel}=10^{-14}$ and a limit of
200 iterations. Table~\ref{tab:cap_counterexample} gives the results for
\eqref{eq:cap_counterexample}; counts include the two initial endpoint
evaluations.

\begin{table}[htbp]
\ifshowchanges\color{red}\fi
\centering
\caption{C\# results for the $C^1$ counterexample on $[-3,1]$.
Restoring the exception cap changes termination, although the existing
100-case benchmark does not activate it.}
\label{tab:cap_counterexample}
\small
\begin{tabular}{@{}llrr@{}}\toprule
Rule & Solver status & Evaluations & Final width\\\midrule
Unlimited exceptions & \texttt{MaxIterationsExceeded} & 202 & $1$\\
Cap $M=3$ & \texttt{Success} & 77 & $5.52\cdot10^{-15}$\\\bottomrule
\end{tabular}
\end{table}

The uncapped implementation exhausts 200 iterations and returns
\texttt{NaN}. With the cap, the returned point is approximately
$-7.776095734362668\cdot10^{-16}$ and the final width is
$5.516726238025032\cdot10^{-15}$. The trace contains 75 completed bracket
updates followed by a width-based exit on the next loop visit; loop visits
and function evaluations are therefore not interchangeable counts.
The first forced fallback occurs after the sixth AB step: steps 3, 4 and 5
have consumed the three permitted residual exceptions. This agrees with
the rule's stated order of tests.

We also ran all 100 existing benchmark functions with the unpatched and
patched implementations and with tracing copies of each, for 400 solves.
For every case the returned value, solver status and actual callback count
agree across the four runs. The cap does not trigger fallback on those
100 functions. Both implementations give 1626 evaluations on the 92
reference-eligible cases, of which the same 89 cases pass independent
reference validation, totaling 1573 evaluations. The three rounded-zero
failures identified in Section~\ref{subsec:rounded_zero_validation} remain
unchanged. The tracing copies also agree with the C\# implementations
on the added counterexample. The existing benchmark therefore does not establish that the bound on
consecutive exceptions can be omitted. These comparisons concern
function evaluation counts; they do not measure the additional execution
time associated with maintaining the counter.
\end{revision}

\subsection{Independent validation of rounded zeros}
\label{subsec:rounded_zero_validation}

The pinned repository validator ended with \texttt{return F(x)==0} after its
reference-distance test had failed. Removing that fallback changes three
classifications in the C\# comparison:
\begin{table}[htbp]
\centering
\caption{Computed zeros rejected by the corrected reference validator.
All three rules return these same inaccurate values.}
\label{tab:false_machine_zeros}
\begin{tabular}{@{}lrrr@{}}\toprule
Case & Returned point & Mathematical zero & Absolute error\\\midrule
f27 &1.88916015625 &1.88889 &$2.70\cdot10^{-4}$\\
f29 &3.000390625 &3 &$3.91\cdot10^{-4}$\\
f75 &0.3333333134651184 &$1/3$ &$1.99\cdot10^{-8}$\\\bottomrule
\end{tabular}
\end{table}

For f27, the evaluated expression is the expanded form of
$(x+1.11111-3)^4\operatorname{sign}(x+1.11111-3)$ and loses its small value by
cancellation. For f29, $\exp((x-3)^5)-1$ rounds to zero before $x$ is close
enough to 3. For f75, cancellation occurs in
\[
 \operatorname{sign}(3x-1)
 \left(1-\sqrt{1-(3x-1)^2/81}\right).
\]
The known root locations reveal the errors independently of the callback.
The revised accuracy check identifies these inaccurate approximations.
It leaves the computed values and the solvers' termination at computed
zeros unchanged.

The proposed validator passed 26 C\# checks covering these three cases,
underflowed callback zeros, exact reference matches, reversed bounds,
non-finite candidates, unavailable or unusable references, relative tolerance,
invalid tolerances, and overflow. A callback that throws if invoked verifies
that validation itself does not evaluate the function. Solver status, external
validation status, and raw function counts remain separate fields in the
record. The minimal repository patch changes \texttt{Problem.IsRoot} and
passes both tolerances at its benchmark call; a separate audit driver supplies
the separate evaluation counts and validation results in
Table~\ref{tab:adaptive_counts}.

\subsection{Reproduction of the analytical and numerical checks}
\label{subsec:checks}

Exact-rational calculations check the iteration bounds used in the proof.
Floating-point examples examine the stopping conditions, the separation
of actual and modified ordinates, and the arithmetic safeguards. These
checks can be reproduced independently of the complete benchmark collection.

For the first check, normalize $W_0=1$ and use
\[
 C\in\{3/2,2,3,8,16,32,33,64,256\},\qquad n=1,\ldots,300.
\]
The upper bounds on the interval length can be computed recursively.
For each number of completed iterations, distinguish intervals obtained
by bisection, intervals whose last AB step requires a return to
bisection, and intervals within a continuing AB phase. In the last
case, record separately for each phase length $t$ the largest possible
initial length $s$. Initially there is only the bisection case, with
interval length $1$. Halving the larger of the bounds in the first two
cases gives the next bisection bound and the initial bound for a possible
new AB phase.

After $t$ steps, a continuing AB phase has interval length at most
$s\min(1,C/2^t)$. If it continues for one more step, the corresponding
bound is $s\min(1,C/2^{t+1})$. The last AB step before a return to
bisection may retain the preceding bound, but only if that bound is
strictly larger than $Cs/2^{t+1}$. When several sequences of iterations
lead to the same case, retain the largest interval bound, or the largest
$s$ for the same phase length. Using exact rational arithmetic, compare
the resulting largest bound after each iteration with
$2^{-\lfloor n/(\lceil\log_2 C\rceil+2)\rfloor}$.
All 2700 inequalities held in the earlier check of the reference rule.
That calculation concerns the geometric bound on interval length;
Lemma~\ref{lem:adaptive_envelope} establishes its applicability to the
adaptive method.
These finitely many comparisons supplement the proof of the iteration
bound.

For the floating-point tests, use
\begin{equation}
 f_p(x)=x^p-10^{-3},\quad [a_0,b_0]=[-10,10],\quad
 p\in\{3,5,7,9,11,13,15,17\}.
 \label{eq:odd_family}
\end{equation}
The unique real root is $x_p^\ast=10^{-3/p}$, and
$f_p'(x_p^\ast)=p(x_p^\ast)^{p-1}>0$.  These are high odd-degree problems
with simple roots.  For the earlier fixed-delay reference calculation, use binary64
round-to-nearest arithmetic, $(B,M)=(16,0)$ and hence $C=32$, $\varepsilon_{\rm abs}=\varepsilon_{\rm rel}=10^{-14}$,
and a limit of 200 updates.  Begin with bisection and use the fixed-delay version defined in
Section~\ref{sec:method}.  Check the returned point against
$x_p^\ast$ using \eqref{eq:external_validation}, rather than accepting
a computed zero as an independent accuracy test.
Internal reference calculations in Python 3.12.14 on Linux passed this
test for all eight cases.  The reference root was evaluated as
$(10^{-3})^{1/p}$ with the platform power function; this is a numerical
reference comparison, not a rigorous enclosure of the real root.
The C\# evaluation counts in Section~\ref{subsec:native_adaptive} are separate
executions; they are not inferred from these Python calculations. No
cross-method timing advantage is claimed.

Three arithmetic cases make the safeguards directly testable.  With
$a=-\Omega/2$, $b=\Omega$, $x=\Omega/4$, $\varepsilon_{\rm abs}=\Omega$ and
$\varepsilon_{\rm rel}=1$, the real width is $3\Omega/2$ and the tolerance
is $5\Omega/4$: the width test must fail even though both unscaled computations
overflow.  With $a=b=2^{-1074}$, the safeguarded midpoint must remain equal
to that endpoint.  Finally, for
$\widehat F(x)=10^{-320}(x-\tfrac14)^3$ at $x=0.26$, a machine zero must
not pass external root-location validation at tolerance $10^{-14}$.
The internal checks reproduced the unguarded defects and passed with the
stated safeguards.

The remaining small checks use $f(x)=x-\tfrac14$ with endpoint pairs
$(\tfrac14,1)$ and $(1,\tfrac14)$, which must return the endpoint root;
$f(x)=x$ on $[-1,1]$ with its midpoint evaluation replaced by
\texttt{NaN}, which must report invalid input; and
$f(x)=s(x^3-2)$ on $[0,2]$ for $s=1$ and $s=10^{120}$.
In the last pair, the two returned values differed by at most $10^{-13}$.
These checks passed in the same internal reference calculation.

\subsection{Implementation and supporting data}

The C\# implementation is maintained by the second author in the
development repository
\url{https://github.com/Proektsoft-EOOD/ModAB-Root-Finding}.
Sections~\ref{sec:method} and \ref{sec:implementation} specify the method,
and Section~\ref{subsec:checks} specifies the reported checks.
The internal checking scripts are not journal supplementary material.
The authors can provide the implementation snapshot and supporting
verification records privately to the editor or reviewers on request.

Reproducing the algorithm and matching a particular sequence of binary64
evaluations are different tasks.  For the latter, the supplied snapshot
must identify the repository commit, runtime and math-library versions,
evaluation order, tolerances, update limit, and callback-counting rule.
The C\# comparison is tied to the immutable commit cited above and to
the separately retained validator patch and audit driver. The reference data,
raw and validated outputs, build information and source checksums are retained
with those records. They can be supplied without adding implementation files
to the compact manuscript project.

\begin{revision}
The separate residual-cap verification package identifies commit
\texttt{51a6d1b} and contains the minimal
patch, C\# test driver, instrumented copies and raw results, including the
counterexample and arithmetic checks. Its reproduction script verifies
the commit and an unmodified source tree, builds in a temporary copy and
checks agreement of the four benchmark variants. These records are
retained separately under \texttt{Verification/ResidualCap\_20260923}
and are available from the authors on request; they are not included in
the compact manuscript archive or needed to compile the manuscript.
They remain separate from the earlier controlled
comparison and from the additional extreme-value safeguards.
\end{revision}

\section{Discussion}
\label{sec:discussion}

An interval can retain endpoint values of opposite signs without its
length tending to zero. Interiority of the interpolation point and
positivity of the ordinate correction factors preserve the sign
condition, but these properties are not sufficient for convergence of the intervals. Ordinary
false position already illustrates this difficulty: one endpoint can
remain fixed. The geometric bound during each AB phase, together with
the bisection that follows its end, ensures the required contraction.

The switching criterion affects the observed convergence rate, but the
global proof does not depend on its particular form. Similarly, the
Anderson--Björck multiplier influences local convergence, whereas the
global argument requires only its positivity and the separation of the
actual values of the function from the modified ordinates.

The estimate \eqref{eq:C32_bound} gives a conservative bound for the
hybrid method. It does not describe its local order of convergence.
The controlled C\# experiment shows an improvement on the specified
finite set of problems and identifies errors in the earlier validation
procedure, but it does not establish that one method is faster on all
admissible functions.

For the adaptive rule, the first two width conditions are satisfied
automatically, and at most three consecutive exceptions are then allowed
when the residual decreases sufficiently. Lemma~\ref{lem:adaptive_envelope}
therefore gives the same constant $C=32$ as for the fixed-delay version.
The last AB step of a phase may violate this bound, but it cannot
increase the length of the interval and the following bisection halves that length.
Accounting for both steps gives the factor $7$ in the global estimate.
The residual factor $1/2$ is a dimensionless parameter choice. Its
optimality has not been proved, and every fixed $\eta\in(0,1)$ gives
the same width estimate when $B$ and $M$ are unchanged.

\begin{revision}
The two exact examples explain the role of the residual condition and
the need to limit its use. In the first, one endpoint converges
quadratically while the length of the interval tends to a positive limit, so
a condition based only on length can interrupt rapid local convergence.
In the second, the residual decreases by a fixed factor at every
subsequent AB step, yet the interval again fails to shrink to a point.
This function satisfies the stated regularity assumptions and involves
no extreme magnitudes. The failure therefore occurs in exact arithmetic
and is not caused by rounding. In the C\# experiment it leads to
termination at the iteration limit despite continued residual decrease.
A finite $M$ restores the condition needed for convergence of the
interval lengths; \texttt{MaxIterations} only bounds the duration of the
calculation. The choice $M=3$ retains the existing $C=32$ estimate and
the observed evaluation counts on the 100 test problems. Other finite
values are allowed by the theory, and this benchmark does not establish
which is optimal.
\end{revision}

The distinction between actual values of the function and modified ordinates
also makes the mathematical method independent of a particular
implementation. The proof concerns the interpolation points and the
successive enclosing intervals; it does not depend on the names of
C\# fields or the software framework.

The convergence theorem assumes a continuous real-valued function and
exact arithmetic. A supplied computer routine may instead evaluate a
discontinuous expression, return \texttt{NaN}, overflow, underflow to
zero on an interval, or depend on an internal state. The arithmetic
safeguards keep the trial points within the numerical interval and
reject \texttt{NaN}, but they cannot ensure that a computed sign agrees
with the sign of the underlying real function. Infinite function
values also fall outside Assumption~1.

Finite precision also limits the possible contraction of the interval.
When its endpoints are adjacent representable binary64 numbers, the
computed midpoint may coincide with an endpoint. The exact-arithmetic
halving identity then no longer describes the computed step. The
implementation therefore uses a positive mixed absolute-relative
tolerance and an iteration limit. The theorem establishes convergence
of the mathematical iteration; the arithmetic safeguards specify the
treatment of exceptional values in its implementation.

\section{Conclusions}
\label{sec:conclusions}

Our corrections to the modified Anderson--Björck method ensure that the
AB rescaling changes only the interpolation ordinates, while the stored
actual function values remain unchanged.
Keeping these quantities separate allows the actual values of the function
to determine the enclosing interval and the modified ordinates to
determine the next approximation.

Global convergence follows by considering the number of bisection
steps. Infinitely many bisections force the interval length to zero.
If there are only finitely many, the final AB phase continues
indefinitely, and Lemma~\ref{lem:adaptive_envelope} bounds its interval
length by a geometric sequence tending to zero. Counting the steps
in completed phases and in a possibly incomplete final phase gives,
for $C=32$,
\[
 W_n\leq W_0\,2^{-\lfloor n/7\rfloor}.
\]
This estimate holds for the iteration in exact arithmetic. Its use as
an error bound for computed approximations requires the additional
conditions on function evaluation and rounding discussed above.

The residual condition allows interpolation to continue when an endpoint
approximation improves rapidly while the interval contracts slowly.
The first exact example demonstrates the benefit of this choice.
Bounding the number of consecutive exceptions preserves the global
estimate for the interval length.
\rev{The second $C^1$ example proves that removing this bound can prevent
the interval length from tending to zero even when the residual
decreases by more than one half at every subsequent AB step. Restoring
the bound of three consecutive exceptions removes the observed
iteration-limit failure without changing the results on the existing
100-case benchmark. The cubic bound for the switching expression
justifies a separate simplification of its evaluation; it does not
replace the condition needed for contraction of the interval.}

Finally, the size of an iteration step, the residual, and the error in
the root approximation are distinct quantities. A small step may be
caused by stagnation, a small residual by scaling, and a computed zero
by rounding. A short interval bounds the distance to a zero only if
that interval actually contains one. The independent validation
procedure checks the approximation against reference roots, while
retaining the original function evaluation counts. This distinction
is also necessary in comparisons between methods, where an inaccurate
result obtained after few evaluations must not be counted as an
efficient solution.

\authorcontributions{Conceptualization, L.T. and N.G.; methodology, L.T. and
N.G.; software, N.G. and L.T.; validation, N.G. and L.T.; formal analysis,
L.T.; investigation, L.T. and N.G.; writing---original draft preparation,
L.T.; writing---review and editing, N.G. and L.T. }

\funding{This research received no external funding.}

\institutionalreview{Not applicable.}

\informedconsent{Not applicable.}

\dataavailability{The test functions, parameters, validation criteria and
rule-check procedure are specified in the manuscript.  The C\#
implementation is maintained in the second author's repository at
\url{https://github.com/Proektsoft-EOOD/ModAB-Root-Finding}.
The Lean formalization is publicly available in the
\href{https://github.com/Proektsoft-EOOD/ModAB-Root-Finding/tree/main/MDPI-Algorithms-Data-2/ModAB.Switching.Lean}{\texttt{ModAB.Switching.Lean} project}.
Additional supporting verification records and the corresponding implementation
snapshot are available from the authors on reasonable request; they can be
provided privately to the editor or reviewers during peer review.}

\acknowledgments{The authors thank the contributors to the public
implementation and benchmark repository.}

\paragraph{Generative-AI assistance.}
ChatGPT/Codex was used to assist with manuscript editing and language
revision, proof checking, Lean formalization, source-code review and edits, and preparation of
test scripts and submission materials.  Exact version information for all
earlier sessions is not available in the authors' records.  

\conflictsofinterest{N.G. is affiliated with Proektsoft EOOD, which maintains
the software repository discussed in this work.  The authors declare no
other conflict of interest.}

\end{document}